%% file: widely_generalized_line_graphs,arXiv.tex
\documentclass[reqno,12pt]{amsart}
\usepackage{url}
\usepackage{amsmath,amsthm,amsfonts,amssymb}
\usepackage{amscd}
\usepackage[dvipdfmx]{graphicx}
\usepackage{colortbl}
\usepackage{comment}

\makeatletter
\@addtoreset{equation}{section}

\makeatother
\newtheorem{thm}{Theorem}[section]
\newtheorem{lem}[thm]{Lemma}
\newtheorem{prop}[thm]{Proposition}
\newtheorem{cor}[thm]{Corollary}
\newtheorem{clm}[thm]{Claim}
\theoremstyle{definition}
\newtheorem{dfn}[thm]{Definition}
\newtheorem{rem}[thm]{Remark}

\newtheorem{ex}[thm]{Example}

\newcommand{\hind}[2]{{\langle \langle #1 \rangle \rangle_{#2}}}
\newcommand{\tildehind}[2]{{\langle \langle #1 \rangle \tilde{\rangle}_{#2}}}
\newcommand{\ind}[2]{{\langle #1 \rangle_{#2}}}

\newcommand{\fakethm}[2]{\noindent {\bf Theorem~\ref{#1}} ~ ~ \emph{#2}}

\DeclareMathOperator{\h}{\mathfrak{h}}
\DeclareMathOperator{\g}{\mathfrak{g}}
\DeclareMathOperator{\m}{\mathfrak{m}}
\DeclareMathOperator{\n}{\mathfrak{n}}
\DeclareMathOperator{\ft}{\mathfrak{t}}

\DeclareMathOperator{\fH}{\mathfrak{H}}
\DeclareMathOperator{\fO}{\mathfrak{O}}

\DeclareMathOperator{\cX}{\mathcal{X}}
\DeclareMathOperator{\cY}{\mathcal{Y}}
\DeclareMathOperator{\sL}{\mathsf{L}}
\DeclareMathOperator{\Aut}{Aut}
\DeclareMathOperator{\Ker}{Ker}
\DeclareMathOperator{\id}{id}

\usepackage{pgf}
\usepackage{tikz}
\usetikzlibrary{backgrounds}
\tikzstyle{slim}=[circle, draw, fill=black, inner sep=0pt, minimum width=8pt]
\tikzstyle{slim_small}=[circle, draw, fill=black, inner sep=0pt, minimum width=6pt]
\tikzstyle{slim_char}=[circle, draw, fill=white, inner sep=1pt, minimum width=6pt]
\tikzstyle{fat}=[circle, draw, fill=black, inner sep=0pt, minimum width=20pt]
\tikzstyle{fat_char}=[circle, draw, fill=white, inner sep=0pt, minimum width=20pt]
\tikzstyle{fat_small}=[circle, draw, fill=black, inner sep=0pt, minimum width=16pt]
\input{./HoffmanGraphs.tex}

\title[The uniqueness of covers for widely generalized line graphs]{The uniqueness of covers of widely generalized line graphs}
\author{Michitaka Furuya}
\address{Kitasato University, College of Liberal Arts and Sciences, 1-15-1~Kitasato, Minami-ku, Sagamihara, Kanagawa, 252-0373, Japan}
\email{michitaka.furuya@gmail.com}
\author{Sho Kubota}
\address{Tohoku University, Graduate School of Information Sciences, 6-3-09~Aoba, Aramamaki-aza, Aoba-ku, Sendai, Miyagi, 980-8579, Japan}
\email{kubota@ims.is.tohoku.ac.jp}
\author{Tetsuji Taniguchi}
\address{Hiroshima Institute of Technology, Department of Electronics and Computer Engineering, 2-1-1~Miyake, Saeki-ku, Hiroshima, 731-5193, Japan}
\email{t.taniguchi.t3@cc.it-hiroshima.ac.jp}
\author{Kiyoto Yoshino}
\address{Tohoku University, Graduate School of Information Sciences, 6-3-09~Aoba, Aramamaki-aza, Aoba-ku, Sendai, Miyagi, 980-8579, Japan}
\email{kiyoto.yosino.r2@dc.tohoku.ac.jp}
\thanks{M. F. was supported by JSPS KAKENHI; grant number: 18K13449,
S.~K. was supported by JSPS KAKENHI; grant number: 18J10656 and
T. T. was supported by JSPS KAKENHI; grant number: 16K05263
 }
 
\date{}

\begin{document}

\maketitle

\begin{abstract}
As a natural generalization of line graphs, Hoffman line graphs were defined by Woo and Neumaier.
Especially, Hoffman line graphs are closely related to the smallest eigenvalue of graphs, and the uniqueness of strict covers of a Hoffman line graph plays a key role in such a study.
In this paper, we prove a theorem for the uniqueness of strict covers under a condition which can be checked in finite time.
Our result gives a generalization and a short proof for the main part of [Ars Math.~Contemp. \textbf{1} (2008) 81--98].
\end{abstract}

\section{Introduction}

Throughout this paper, we consider only finite undirected graphs without loops or multiple edges.

For a graph $G$, the {\it line graph} $\sL(G)$ of $G$ is the graph obtained by $V(\sL(G))=E(G)$ and $E(\sL(G))=\{\{e,e'\} : e,e' \in E(G), | e\cap e' | = 1 \}$.
Line graphs have an important structure representing claw-free graphs (see \cite{CS}), and many researchers have studied properties on line graphs.
For example, Thomassen~\cite{Tho} conjectured that every $4$-connected line graph is Hamiltonian.
To attack Thomassen's conjecture or related topics, we frequently focus on a graph $G_{L}$ such that $\sL(G_{L})$ is isomorphic to the target line graph $L$ (where such a graph $G_{L}$ is called a {\it preimage} of $L$), and discuss the existence of a closed trail with a good property in $G_{L}$ instead of the existence of a Hamiltonian cycle in $L$.
Thus it is important to analyze the structure of a preimage of a line graph.
However, in general, there exist line graphs having two distinct preimages; for example, the triangle and the claw are distinct preimages of the triangle.
Furthermore, when we consider the difference for the correspondence of the vertices of a line graph to the edges of a preimage, we can construct another example of order $6$ (cf. Figure~\ref{fig:LG_2cov} in Section~\ref{sec:pre}).
On the other hand, it is known that a preimage of a line graph of order at least $7$ is uniquely determined even if we consider above difference (cf. Corollary~\ref{rem:GLG}).

A concept of Hoffman graphs appeared implicitly in \cite{H} and was strictly defined by Woo and Neumaier~\cite{WN} as a natural generalization of line graphs, and such graphs are especially used in algebraic graph theory.
Since the definition of a Hoffman line graph is slightly complicated, we postpone giving its strict definition and related notations until Section~\ref{sec:pre} and only give brief descriptions here.
(Thus the readers who want to know strict significance of our result are advised to previously read Section~\ref{sec:pre}.)
We describe typical Hoffman graphs in Figure~\ref{fig:Hoffmangraphs}, where their names derive from a traditional custom.
It is a worthy fact that Hoffman graphs have ``slim'' vertices and ``fat'' vertices.
\begin{figure}	[htbp]
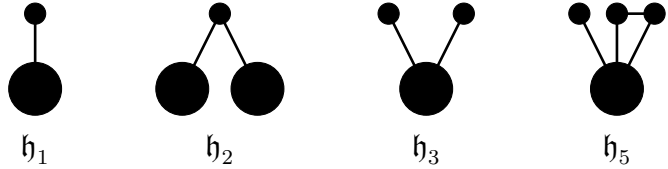
	\label{fig:Hoffmangraphs}
	\centering
	\Hoffa{30pt} \qquad \Hoffb{30pt} \qquad  \Hoffc{30pt} \qquad \Hoffe{30pt}
	\caption{Hoffman graphs}
\end{figure}
For a family $\fH$ of Hoffman graphs, Woo and Neumaier~\cite{WN} defined a {\it slim $\fH$-line graph} and its {\it $\fH$-cover}.
Observing the definition of Hoffman line graphs, we can verify that a given slim $\{\h_{2}\}$-line graphs and its strict $\{\h_{2}\}$-covers correspond to original line graphs and their preimages, respectively.

The uniqueness of strict $\fH$-covers is frequently used in the study of slim $\fH$-line graphs, and the following result is known.

\begin{thm}[Cvetkovi\'c, Doob and Simi\'c~\cite{CDS2}]
\label{thm:GLG}
Every connected slim $\{\h_2,\h_3\}$-line graph of order at least $7$ has a unique strict $\{\h_2,\h_3\}$-cover up to equivalence.
\end{thm}

As a corollary of Theorem~\ref{thm:GLG}, we obtain the following result which assures the uniqueness of preimages of a large line graph.

\begin{cor}
\label{rem:GLG}
Every connected slim $\{\h_2\}$-line graph of order at least $7$ has a unique strict $\{\h_2\}$-cover up to equivalence.
\end{cor}

Taniguchi~\cite{T1} focused on slim $\{\h_2,\h_5\}$-line graphs from the viewpoint of a characterization of graphs with the smallest eigenvalue at least $-1-\sqrt{2}$, and he gave an analogy of Theorem~\ref{thm:GLG} and Corollary~\ref{rem:GLG}:
Every connected slim $\{\h_2,\h_5\}$-line graph of order at least $8$ has a unique strict $\{\h_2,\h_3,\h_5\}$-cover up to equivalence.
The proof of this result was mainly spent on the following theorem, and he found an integer $N$ satisfying Theorem~\ref{thm:T1} as $N=8$ by using computer search.

\begin{thm}[Taniguchi~\cite{T1}]
\label{thm:T1}
Let $N\geq 7$ be an integer.
If every connected slim $\{\h_2,\h_5\}$-line graph of order $N$ has exactly one strict $\overline{\{\h_2,\h_5\}}$-cover up to equivalence,
then every connected slim $\{\h_2,\h_5\}$-line graph of order at least $N$ has exactly one strict $\overline{\{\h_2,\h_5\}}$-cover up to equivalence.
\end{thm}

In this paper, we give the following generalization of Theorems~\ref{thm:T1} with an alternative (and short) proof.
The symbols $\fO$ and $\bar{\fH}$ in the following theorem is defined in Definition~\ref{dfn:OandbarH}.

\begin{thm}
\label{thm:uni}
Let $\fH \subset \fO$ be a family with $\h_2 \in \fH$, and let $N\geq 7$ be an integer.
If every connected slim $\fH$-line graph of order $N$ has exactly one strict $\bar{\fH}$-cover up to equivalence, then every connected slim $\fH$-line graph of order at least $N$ has exactly one strict $\bar{\fH}$-cover up to equivalence.
\end{thm}

This paper is organized as follows:
In Section~\ref{sec:pre}, we define the concept of Hoffman graphs and related topics.
In Section~\ref{sec:lem}, we give some lemmas which are used in the argument for Hoffman graphs.
Many lemmas in Subsection~\ref{sec:lem-sum} have been used in some existing research as folklore.
However, to keep the paper self-contained, we give their proof (and so the proof of some lemmas in Subsection~\ref{sec:lem-sum} does not affect to the shortness of our proof).
Hence readers familiar with Hoffman graphs are advised to skip the proof.
We give more essential lemmas for our proof in Subsections~\ref{sec:connected} and \ref{sec:order}.
We prove Theorem~\ref{thm:uni} in Section~\ref{sec:cov}.
In Section~\ref{sec:computer}, we prove further propositions used in computer search, and we also demonstrate computer search to find the existence of $N$ in Theorem~\ref{thm:uni} for a family of $\fH$ other than $\{\h_{2},\h_{5}\}$ (see Example~\ref{ex: 6.7}).

\section{Hoffman graph}\label{sec:pre}

In this section, we define Hoffman graphs and related concepts.

\begin{dfn}[{\bf Hoffman graph}]
\label{dfn:Hoff}
A {\it Hoffman graph} $\h$ is a pair $(H,\mu)$ of a graph $H$ and a labeling map $\mu:V(H) \to \{f,s\}$, where $V(H)$ denotes the vertex set of $H$, satisfying the following conditions:
\begin{enumerate}
\item
Every vertex with label $f$ is adjacent to at least one vertex with label $s$; and
\item
the vertices with label $f$ are pairwise non-adjacent.
\end{enumerate}
\end{dfn}

Several symbols defined below are analogous to ones used in graph theory.
Let $\h=(H,\mu)$ be a Hoffman graph.
The vertices of $H$ are regarded as the vertices of $\h$.
A vertex of $H$ with label $s$ (resp.\ label $f$) is called a {\it slim vertex} (resp.\ a {\it fat vertex}).
We let $V_s(\h)$ (resp.\ $V_f(\h)$) denote the set of slim vertices (resp.\ fat vertices) of $H$, and let $V(\h)=V_s(\h)\cup V_f(\h)$.
We let $E(\h)$ denote the set of edges of $H$.	
For a vertex $x$ of $\h$, we let $N_{\h}^{s}(x)$ (resp.\ $N_{\h}^{f}(x)$) denote the set of neighbors labeled $s$ (resp.\ $f$) of $x$, and set $N_{\h}(x)=N_{\h}^s(x) \cup N_{\h}^f(x)$.
For two vertices $x$ and $y$, we write $x \sim y$ if $x \in N_{\h}(y)$.
A Hoffman graph $\h=(H,\mu)$ is called a {\it slim graph} if $\h$ has no fat vertices, i.e., $\mu (x)=s$ for every vertex $x$ of $\h$.
We regard an ordinary graph with no labeling as a slim graph.
A Hoffman graph is said to be {\it fat} if every slim vertex is adjacent to a fat vertex.
A Hoffman graph $\h'=(H',\mu')$ is called an {\it induced (Hoffman) subgraph} of $\h$ if $H'$ is an induced subgraph of $H$ and $\mu|_{V(H')}=\mu'$.	
The rest of this paper, ``(Hoffman) subgraph'' means ``induced (Hoffman) subgraph''.
For $X\subset V(\h)$, let $\ind{X}{\h}$ denote the Hoffman subgraph of $\h$ induced by $X$, that is, the pair of the subgraph of $H$ induced by $X$ and the labeling map $\mu|_{X}$.
The graph $\ind{V_s(\h)}{\h}$ is called the {\it slim subgraph} of $\h$.
For $X\subset V_s(\h)$, let $\hind{X}{\h}$ denote the Hoffman subgraph of $\h$ induced by $X \cup ( \bigcup_{x \in X} N_{\h}^f(x) )$.
A (Hoffman) graph is {\it empty} if it has no vertices.

Next we give the definition of the sum of Hoffman graphs.
The definition may seem to be strange, but in fact it comes from lattices, which is described in~\cite{WN}.

\begin{dfn}[{\bf Sum of Hoffman graphs}]
\label{definition:sum}
Let $\h$ be a Hoffman graph, and let $\h^1$ and $\h^2$ be Hoffman subgraphs of $\h$.
We say that $\h$ is the {\it sum} of $\h^1$ and $\h^2$, denoted by $\h=\h^1 \oplus \h^2$, if the following conditions hold:
\begin{enumerate}
\item	\label{s1}
$V(\h)=V(\h^1) \cup V(\h^2)$;
\item	\label{s2}
$V_s(\h)=V_s(\h^1) \sqcup V_s(\h^2)$;
\item	\label{s3}
for $i \in \{1,2\}$ and $x \in V_s(\h^i)$, $N_{\h^i}^f(x) = N_{\h}^f(x)$;
\item	\label{s4}
for $x \in V_s(\h^1)$ and $y \in V_s(\h^2)$, $|N_{\h}^f(x)\cap N_{\h}^f(y)|\leq 1$; and
\item	\label{s5}
for $x \in V_s(\h^1)$ and $y \in V_s(\h^2)$, $|N_{\h}^f(x)\cap N_{\h}^f(y)|=1$ if and only if $x\sim y$ in $\h$.
\end{enumerate}
Note that a Hoffman graph $\h$ can be regarded as the sum of $\h$ and the empty Hoffman graph.
If $\h$ is the sum of two non-empty Hoffman graphs, then it is said to be {\it decomposable}; otherwise, it is said to be {\it indecomposable}.
Note that the sum of Hoffman graphs satisfies commutative and associative law.
Thus the sum of more than two Hoffman graphs is naturally defined, and the sum of only one Hoffman graph should be itself.
An example of a decomposable Hoffman graph are depicted in Figure~\ref{fig:K00}.
For convenience, for a set $\{\h_i\}_{i \in I}$ of Hoffman graphs, we let $\bigoplus_{i \in I} \h_i$ denote the empty graph if $I$ is empty.
For a Hoffman graph $\h$, a non-empty Hoffman subgraph $\h^1$ of $\h$ is called an {\it addend} of $\h$ if there exists a Hoffman subgraph $\h^2$ of $\h$ such that $\h = \h^1 \oplus \h^2$.
We can regard a non-empty Hoffman graph $\h$ as an addend of $\h$.
\end{dfn}

\begin{figure}[htpb]
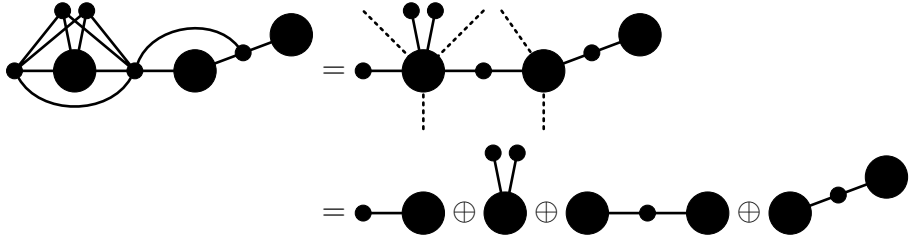

\begin{align*}
	\ExSumWithEdges &=	\ExSum \\
	&= \ExSuma \oplus \ExSumb \oplus \ExSumc  \oplus \ExSumd
\end{align*}
\caption{Example of the sum of Hoffman graphs, whose slim (resp. fat) vertices are depicted as small (resp. large) filled circles, where every region delimited by dotted lines represents an indecomposable addend.}
\label{fig:K00}
\end{figure}

Let $\h=(H,\mu)$ and $\h'=(H',\mu')$ be Hoffman graphs.
A graph isomorphism $\varphi$ from $H$ to $H'$ is called an {\it isomorphism} from $\h$ to $\h'$, written by $\varphi: \h \to \h'$, if $\varphi$ preserves the fatness and the slimness of vertices (i.e., $\varphi (V_s(\h))=V_s(\h')$ and $\varphi (V_f(\h))=V_f(\h')$).
In addition, for an isomorphism $\varphi : \h \to \h'$ and a Hoffman subgraph $\n$ of $\h$, let $\varphi |_{\n}$ denote the restriction $\varphi |_{V(\n)} : V(\n) \to \varphi(V(\n))$, and $\varphi(\n)$ denote the subgraph in $\h'$ induced by $\varphi(V(\n))$.
The Hoffman graphs $\h$ and $\h'$ are {\it isomorphic}, denoted by $\h \simeq \h'$, if there exists an isomorphism from $\h$ to $\h'$.
For a Hoffman graph $\h$ and a family $\fH$ of Hoffman graphs, we write $\h \in \fH$ if $\h$ is isomorphic to a Hoffman graph in $\fH$.

\begin{dfn}[{\bf Families $\fO$ and $\bar{\fH}$}]
\label{dfn:OandbarH}
Let $\fO$ be the family consisting of $\h_2$ and the indecomposable fat Hoffman graphs $\h$ such that $|V_s(\h)|\geq 2$ and $|V_f(\h)|=1$.
For a family $\fH \subset \fO$, we let
$$
\bar{\fH}=\{\h_2\}\cup \{\h\in \fO:\h \text{ is a Hoffman subgraph of an element of }\fH\}.
$$
\end{dfn}

\begin{dfn}[{\bf Line Hoffman graph}]
\label{definition:H-line}
Let $\fH$ be a family of Hoffman graphs.
A Hoffman graph $\g$ is called an {\it $\fH$-line Hoffman graph} if $\g$ is a Hoffman subgraph of a Hoffman graph $\h = \bigoplus_{i=0}^n \h^i$ where $\h^i \in \fH$ for every $i$.
In the above situation, $\h$ is called an {\it $\fH$-cover} of $\g$.
For an $\fH$-line Hoffman graph $\g$, an $\fH$-cover $\h$ of $\g$ is said to be \emph{strict} if $V_s(\h)=V_s(\g)$.
A slim $\fH$-line Hoffman graph is simply said to be a {\it slim $\fH$-line graph}.
Two strict $\fH$-covers $\h$ and $\h'$ of an $\fH$-line Hoffman graph $\g$ are said to be {\it equivalent} if there exists an isomorphism $\varphi: \h \to \h'$ such that $\varphi|_{\g}$ is the identity mapping $\id_{V(\g)}$.
Note that there exists a graph having two non-equivalent strict $\fH$-covers (see Figure~\ref{K01}).
\end{dfn}

\begin{figure}[hbpt]
\begin{align*}
\ExEquiv && \ExEquiva && \ExEquivb
\end{align*}
\caption{Two non-equivalent strict $\{\h_1,\h_2\}$-covers $\h$ and $\h'$ of a graph $G$}
\label{K01}
\end{figure}

As we depict in Figure~\ref{fig:LG_2cov}, it is known that there exists a slim $\{\h_2\}$-line graph of order $6$ having two non-equivalent strict $\{\h_2\}$-covers.
Hence, when we discuss the uniqueness of covers of a slim $\fH$-line graph for a family $\fH$ of Hoffman graphs containing $\h_2$, the condition ``order at least $7$'' is necessary (cf. Theorem~\ref{thm:uni}).
\begin{figure}[htbp]
\begin{align*}
\begin{tikzpicture}[
baseline=0,
yscale = 0.92, xscale = 0.8 , e/.style={line width=1 pt, line join=round, line cap=round}
]
\node at (0,-2) [slim] (a) {};
\node at (2,-1) [slim] (b) {};
\node at (2,1) [slim] (c) {};
\node at (0,2) [slim] (d) {};
\node at (-2,1) [slim] (e) {};
\node at (-2,-1) [slim] (f) {};
\draw [e](a)--(b)--(c)--(d)--(e)--(f)--(a);
\draw [e](a)--(c)--(e)--(a) (b)--(d)--(f)--(b);		
\end{tikzpicture}
&&
\begin{tikzpicture}[
baseline=0,
yscale = 0.92, xscale = 0.8, e/.style={line width=0.5 pt, line join=round, line cap=round},
be/.style={line width=2 pt, line join=round, line cap=round}
]
\node at (0,-2) [slim] (a) {};
\node at (2,-1) [slim] (b) {};
\node at (2,1) [slim] (c) {};
\node at (0,2) [slim] (d) {};
\node at (-2,1) [slim] (e) {};
\node at (-2,-1) [slim] (f) {};
\node at (0,0) [fat] (A) {};
\node at (2.6,1.3) [fat] (B) {};
\node at (-2.6,1.3) [fat] (C) {};
\node at (0,-2.6) [fat] (D) {};
\draw [e](a)--(b)--(c)--(d)--(e)--(f)--(a);
\draw [e](a)--(c)--(e)--(a) (b)--(d)--(f)--(b);		
\draw [be](a)--(A)--(c)--(B)--(d)--(C)--(f)--(D)--(a);
\draw [be](B)--(b)--(D) (e)--(C) (e)--(A);
\end{tikzpicture}
&&
\begin{tikzpicture}[
baseline=0,
yscale = 0.92, xscale = 0.8, e/.style={line width=0.5 pt, line join=round, line cap=round},
be/.style={line width=2 pt, line join=round, line cap=round}
]
\node at (0,-2) [slim] (a) {};
\node at (2,-1) [slim] (b) {};
\node at (2,1) [slim] (c) {};
\node at (0,2) [slim] (d) {};
\node at (-2,1) [slim] (e) {};
\node at (-2,-1) [slim] (f) {};
\node at (0,0) [fat] (A) {};
\node at (2.6,-1.3) [fat] (B) {};
\node at (0,2.6) [fat] (C) {};
\node at (-2.6,-1.3) [fat] (D) {};
\draw [e](a)--(b)--(c)--(d)--(e)--(f)--(a);
\draw [e](a)--(c)--(e)--(a) (b)--(d)--(f)--(b);		
\draw [be](f)--(A)--(b)--(B)--(c)--(C)--(e)--(D)--(f);
\draw [be](D)--(a)--(B) (d)--(C) (A)--(d);
\end{tikzpicture}			
\end{align*}
\caption{A line graph and its non-equivalent two strict $\{\h_2\}$-covers}
\label{fig:LG_2cov}
\end{figure}
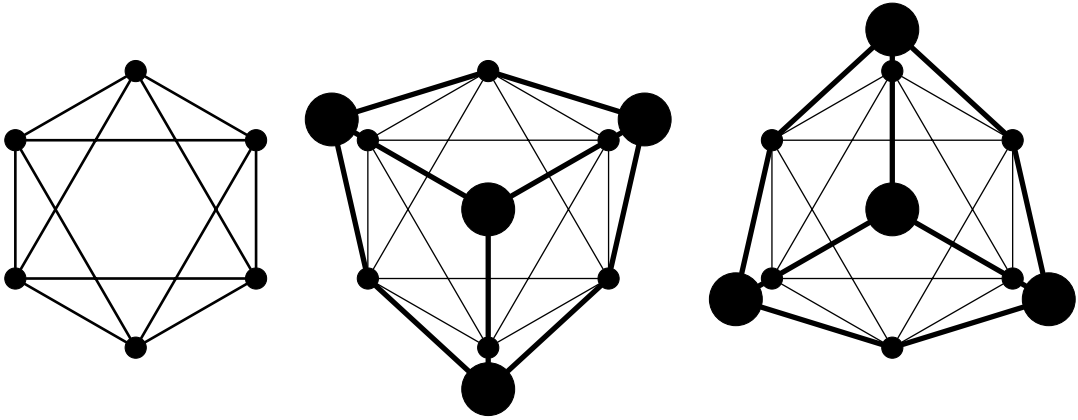

\section{Basic properties and lemmas for Hoffman graphs}\label{sec:lem}

\subsection{Sum of Hoffman graphs}\label{sec:lem-sum}

In this subsection, we discuss a uniqueness and an expression for the sum of Hoffman graphs via indecomposable addends.

\begin{lem}
\label{lem:irr0}
For a Hoffman graph $\n$, an indecomposable decomposition of $\n$ is uniquely determined, that is, if
$$
\n = \bigoplus_{i=0}^k \n^i = \bigoplus_{i=0}^l \m^i
$$
for indecomposable Hoffman graphs $\n^0,\ldots,\n^k,\m^0,\ldots,\m^l$, then $k=l$ and there exists a permutation $\sigma$ on $\{0,1,\ldots ,k\}$ such that $\n^i = \m^{\sigma(i)}$ for each $i$.
\end{lem}
\begin{proof}
We define the graph $G$ as $V(G) = V_s(\n)$ and
$$
E(G) = \{ \{x,y\} : x,y \in V_s(\n) \text{ with }x\neq y \text{ and } w(x,y) \neq 0 \},
$$
where
$$
w(x,y) := -|N_{\n}^f(x) \cap N_{\n}^f(y)|+
\begin{cases}
1 & \text{if } x \sim y \text{ in } \n \\
0 & \text{otherwise}.
\end{cases}
$$
Let $G_0,\ldots,G_n$ be the connected components of $G$.

\begin{clm}
\label{clm:r1}
For each $j \in \{0,\ldots,n\}$, there exist indices $i_{1} \in \{0,\ldots,k\}$ and $i_{2} \in \{0,\ldots,l\}$ such that $V(G_{j})\subset V_{s}(\n^{i_{1}})$ and $V(G_{j})\subset V_{s}(\m^{i_{2}})$.
\end{clm}
\begin{proof}[Proof of Claim~\ref{clm:r1}]
By the symmetry of $\n^{i}$ and $\m^{i}$, it suffices to show that for two adjacent vertices $x$ and $y$ of $G$, there exists an index $i\in \{0,\ldots,k\}$ with $x,y\in V_{s}(\n^{i})$.
By way of contradiction, we suppose that there exist two adjacent vertices $x$ and $y$ of $G$ such that $x\in V_{s}(\n^{i})$ and $y\in V_{s}(\n^{i'})$ for some indices $i$ and $i'$ with $i\neq i'$.
By the definition of edges of $G$, $w(x,y)\neq 0$.
If $x\sim y$ in $\n$, then it follows from Definition~\ref{definition:sum}~(\ref{s5}) that $|N_{\n}^f(x) \cap N_{\n}^f(y)|=1$, and hence $w(x,y) = -|N_{\n}^f(x) \cap N_{\n}^f(y)|+1 = 0$, which is a contradiction.
Thus $x\not\sim y$ in $\n$.
Then by Definition~\ref{definition:sum}~(\ref{s4}) and (\ref{s5}), we have $|N_{\n}^f(x) \cap N_{\n}^f(y)| =0$.
Thus $w(x,y) = -|N_{\n}^f(x) \cap N_{\n}^f(y)| = 0$, which is a contradiction.
\end{proof}

\begin{clm}
\label{clm:r2}
We have $\{ V(G_0),\ldots,V(G_n) \}=\{ V_s(\n^0),\ldots,V_s(\n^k) \}=\{ V_s(\m^0),\ldots,V_s(\m^l) \}$.
\end{clm}
\begin{proof}[Proof of Claim~\ref{clm:r2}]
By Claim~\ref{clm:r1} and the symmetry of indices, it suffices to show that $|\{j:V(G_{j}) \subset V_{s}(\n^{0})\}|=1$.
Considering Claim~\ref{clm:r1} again, without loss of generality, we may assume that
\begin{align}
V_s(\n^0)=V(G_0) \sqcup \cdots \sqcup V(G_m) \text{ for a non-negative integer } m.
\end{align}
We show that $m=0$.
By way of contradiction, we suppose that $m\geq 1$.
Let $\ft^0 := \hind{V(G_0)}{\n^0}$ and $\ft^1 := \hind{V(G_1) \sqcup \cdots \sqcup V(G_m)}{\n^0}$.

Now we verify the five conditions in Definition~\ref{definition:sum} for $\h=\n^0$, $\h^1=\ft^0$ and $\h^2 = \ft^1$.
Note that Definition~\ref{definition:sum}~(\ref{s3}) is clearly satisfied.
Since $\{V(G_0), \ldots, V(G_m)\}$ is a partition of $V_s(\n^0)$, we have
\begin{align}
V_s(\n^0) = V_s(\ft^0) \sqcup V_s(\ft^1),\label{r2}
\end{align}
and so Definition~\ref{definition:sum}~(\ref{s2}) is satisfied.

Let $z\in V_f(\n^0)$.
Then by the definition of Hoffman graph, there exists a vertex $w\in N_{\n^{0}}^{s}(z)$.
Note that $w$ is a vertex belonging to exactly one of $G_0, \ldots, G_m$.
Hence we obtain either $w \in V_s(\ft^0)$ or $w \in V_s(\ft^1)$.
This together with the definition of the symbol $\hind{\cdot}{{}}$, $z$ belongs to $V_f(\ft^0)$ or $V_f(\ft^1)$.
Since $z$ is arbitrary, we have $V_f(\n^0) \subset V_f(\ft^0) \cup V_f(\ft^1)$, and so $V_f(\n^0) = V_f(\ft^0) \cup V_f(\ft^1)$.
This together with (\ref{r2}) implies that Definition~\ref{definition:sum}~(\ref{s1}) is satisfied.

Let $x \in V_s(\ft^0)$ and $y \in V_s(\ft^1)$.
Since $x\not\sim y$ in $G$, we have
$$
0 = w(x,y) = -|N_{\n}^f(x) \cap N_{\n}^f(y)|+
\begin{cases}
1 & \text{if } x \sim y\text{ in }\n\\
0 & \text{otherwise}.
\end{cases}
$$
Hence $|N_{\n^0}^f(x) \cap N_{\n^0}^f(y)| =|N_{\n}^f(x) \cap N_{\n}^f(y)| \leq 1$, and the equality holds if and only if $x$ and $y$ are adjacent in $\n^0$.
This implies that Definition~\ref{definition:sum}~(\ref{s4}) and (\ref{s5}) are satisfied.

Consequently, we have $\n^0 = \ft^0 \oplus \ft^1$, which contradicts the indecomposability of $\n^0$.
\end{proof}

By Claim~\ref{clm:r2}, $k=l$ and there exists a permutation $\sigma$ on $\{0,1,\ldots ,k\}$ such that
\begin{align}
V_s(\n^i) = V_s(\m^{\sigma(i)}) \text{ for every $i$}.\label{r6}
\end{align}
Furthermore, it follows from Definition~\ref{definition:sum}~(\ref{s3}) that $\n^i = \hind{V_s(\n^i)}{\n}$ and $\m^i = \hind{V_s(\m^i)}{\n}$ for every $i$.
This together with (\ref{r6}) leads to
$$
\n^i = 	\hind{V_s(\n^i)}{\n} = \hind{V_s(\m^{\sigma(i)})}{\n} = \m^{\sigma(i)}.
$$
for every $i$.
This completes the proof of Lemma~\ref{lem:irr0}.
\end{proof}

As a corollary of Lemma~\ref{lem:irr0}, we obtain the following result which claims that every isomorphism $\varphi : \n \to \m$ between Hoffman graphs maps each indecomposable addend of $\n$ to an indecomposable addend of $\m$.

\begin{cor}
\label{cor:irr}
Let
$$
\varphi : \n = \bigoplus_{i=0}^k \n^i \to \m = \bigoplus_{i=0}^l \m^i
$$
be an isomorphism between two Hoffman graphs $\n$ and $\m$ where $\n^0,\ldots,\n^k,\m^0,\ldots,\m^l$ are indecomposable addends.
Then $k=l$ and there exists a permutation $\sigma$ on $\{0,1,\ldots ,k\}$ such that $\varphi|_{\n^i} : \n^i \to \m^{\sigma(i)}$ for every $i$.
\end{cor}
\begin{proof}
By the definition of $\varphi $ and the sum of Hoffman graphs, we have
$$
\bigoplus_{i=0}^k \varphi(\n^i) = \varphi(\n) = \m = \bigoplus_{i=0}^l \m^i,	
$$
i.e., $\m$ has two indecomposable decompositions $\bigoplus_{i=0}^k \varphi(\n^i)$ and $\bigoplus_{i=0}^l \m^i$.
Applying Lemma~\ref{lem:irr0} to the decompositions, we obtain $k=l$ and there exists a permutation $\sigma$ on $\{0,1,\ldots ,k\}$ such that $\varphi (\n^i)=\m^{\sigma(i)}$, that is, $\varphi|_{\n^i} : \n^i \to \m^{\sigma(i)}$ for every $i$.
\end{proof}

\begin{dfn}[{\bf Hoffman graph $\tilde{\n}$}]
\label{definition:tilde}
Let $\n=\bigoplus_{i=0}^k \n^i$ be a Hoffman graph such that $\n^0,\ldots,\n^l \in \{\h_1\}$ and $\n^{l+1},\ldots,\n^k \in \fO$.
If none of $\n^0, \ldots, \n^k$ is isomorphic to $\h_1$, then we define $\tilde{\n} := \n$; otherwise, we let $\tilde{\n}$ denote the Hoffman graph as follows:
\begin{align*}
&V_s(\tilde{\n}) := V_s(\n),\\
&V_f(\tilde{\n}) :=  V_f(\n) \sqcup \{f_0,\ldots,f_l\}, \text{ and }\\
&E(\tilde{\n}) := E(\n) \sqcup \left\{\{s_i,f_i\} \mid i=0,\ldots,l  \right\},
\end{align*}
where $s_i$ is the unique slim vertex of $\n^i$ for $i=0,\ldots,l$ and $f_0,\ldots,f_l$ are pairwise distinct new fat vertices.
In other words, $\tilde{\n}$ is the Hoffman graph obtained from $\n$ by replacing each addend of $\n$ isomorphic to $\h_1$ by a new Hoffman graph isomorphic to $\h_2$.
We often use $(\n \tilde{)}$ instead of $\tilde{\n}$ if the construction (or the formula) of $\n$ is complicated (see Figure~\ref{K02}).
\end{dfn}

\begin{figure}[htbp]
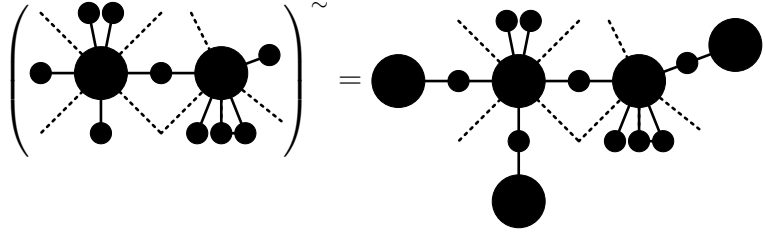

\centering
$\left( \ExOfTilde \right)^{\sim}=\ExOfTildee$
\caption{Example of Definition~\ref{definition:tilde}}
\label{K02}
\end{figure}		

We next give lemmas concerning $\tilde{\n}$.

\begin{lem}
\label{lem:tilde}
Let $\n=\bigoplus_{i=0}^k \n^i$ be a Hoffman graph such that $\n^0,\ldots,\n^l \in \{\h_1\}$ and $\n^{l+1},\ldots,\n^k \in \fO$.
Then
$$
\tilde{\n} = \bigoplus_{i=0}^k (\n^i \tilde{)}.
$$
\end{lem}
\begin{proof}
It suffices to check that the five conditions in Definition~\ref{definition:sum} are satisfied.
Let $f_{0},\ldots ,f_{l}$ be the fat vertices as in the definition of $\tilde{\n}$.
Then it is clear that
$$
V_s(\tilde{\n}) = V_s(\n) = \bigsqcup_{i=0}^k V_s(\n^i) = \bigsqcup_{i=0}^k V_s((\n^i\tilde{)})
$$
and
\begin{align*}
V_f(\tilde{\n}) &= V_f(\n) \sqcup \{f_0,\ldots,f_l\}= \left( \bigcup_{i=0}^k V_f(\n^i)\right)\sqcup \{f_0,\ldots,f_l\}	\\
&= \left( \bigcup_{i=0}^l \left( V_f(\n^i) \sqcup \{f_i\} \right) \right)\cup \left(\bigcup_{i=l+1}^k V_f(\n^i)\right) = \bigcup_{i=0}^k V_f((\n^i\tilde{)}).
\end{align*}
The above equations imply that the conditions (\ref{s1}) and (\ref{s2}) in Definition~\ref{definition:sum} are satisfied.

Let $i \in \{0, \ldots, k\}$ and $x \in V_s(\n^i)$.
If $0\leq i\leq l$, then
$$
N_{\tilde{\n}}^f(x) =N_{\n}^f(x) \cup \{f_i \}= N_{\n^i}^f(x) \cup \{f_i\}= N_{(\n^i\tilde{)}}^f(x);
$$
otherwise,
$$
N_{\tilde{\n}}^f(x)=N_{\n}^f(x)= N_{\n^i}^f(x)= N_{(\n^i\tilde{)}}^f(x).
$$
In either case, we have $N_{(\n^i\tilde{)}}^f(x)=N_{\tilde{\n}}^f(x)$, which implies that the condition (\ref{s3}) in Definition~\ref{definition:sum} is satisfied.

Since each additional fat vertex of $\tilde{\n}$ is adjacent to exactly one slim vertex, the conditions (\ref{s4}) and (\ref{s5}) in Definition~\ref{definition:sum} are satisfied.
\end{proof}

The following lemma was proved in \cite{T1}.

\begin{lem}[{\cite[Lemma~12]{T1}}]
\label{lem:tani}
Let $\h = \bigoplus_{i=0}^k \h^i$ be a Hoffman graph, and let $X\subset V_{s}(\h)$.
Then
$$
\hind{X}{\h} = \bigoplus_{i=0}^k \hind{X \cap V_s(\h^i)}{\h}.
$$
\end{lem}

\begin{lem}
\label{lem:ExiCov}
Let $\fH$ be a subfamily of $\fO$.
Let $\h = \bigoplus_{i=0}^k \h^i$ be a Hoffman graph with $\h^i \in \fH$ for every $i$, and let $G$ be a subgraph of the slim subgraph of $\h$.
Then $\tildehind{V(G)}{\h}$ is a strict $\bar{\fH}$-cover of $G$.
In particular, every slim $\fH$-line graph has a strict $\bar{\fH}$-cover.
\end{lem}
\begin{proof}
For each $i$, let $\n^i = \hind{V(G) \cap V_s(\h^i)}{\h^i}$.
Then by Lemma~\ref{lem:tani},
$$
\hind{V(G)}{\h} = \bigoplus_{i=0}^k \n^i.
$$
Now for each $i \in \{0,\ldots,k\}$, we write
$$
\n^i = \bigoplus_{j \in J_i} \n^{i,j},
$$
where $J_i$ is an index set and $\n^{i,j}$'s are indecomposable Hoffman subgraphs of $\n^i$.
If $|V_{s}(\n^{i,j})|=1$, then $\n^{i,j}$ is isomorphic to either $\h_1$ or $\h_2$; if $|V_{s}(\n^{i,j})|\geq 2$, then $\n^{i,j}\in \bar{\fH}\setminus \{\h_{2}\}$ since $\n^{i,j}$ is an indecomposable Hoffman subgraph of $\hind{V(G) \cap V_s(\h^i)}{\h^i}$ and $|V_{f}(\n^{i,j})|=|V_{f}(\hind{V(G) \cap V_s(\h^i)}{\h^i})|=1$.
In either case, $(\n^{i,j}\tilde{)}\in \bar{\fH}$ for all $i$ and $j$.
Since $\tildehind{V(G)}{\h} =  \bigoplus_{i,j} (\n^{i,j}\tilde{)}$ by Lemma~\ref{lem:tilde}, $\tildehind{V(G)}{\h}$ is a strict $\bar{\fH}$-cover of $G$.
\end{proof}

\subsection{Connectedness}\label{sec:connected}

Let $\h$ be a Hoffman graph, and let $\n$ be a Hoffman subgraph of $\h$.
For $X\subset V(\h)$, we let $\n-X$ denote the Hoffman subgraph of $\n$ induced by $V(\n)\setminus X$.
For $x\in V(\h)$, we let $\n-x=\n-\{x\}$.
For a graph $G$, we let $\bar{G}$ denote the complement of $G$.
For two vertex-disjoint graphs $G$ and $H$, we let $G\sqcup H$ be the graph such that $V(G\sqcup H)=V(G)\cup V(H)$ and $E(G\sqcup H)=E(G)\cup E(H)$.

\begin{lem}
\label{lem:conn}
Let $\n=\n^0 \oplus \n^1$ be a Hoffman graph such that $\n^0 \in \fO \setminus \{ \h_2 \}$ and $\n^{1}$ is non-empty, and suppose that the slim subgraph of $\n$ is connected.
Then the slim subgraph of $\n-x$ is connected for every $x\in V_{s}(\n^{0})$.
\end{lem}
\begin{proof}
Let $G$ be the slim subgraph of $\n$.
Let $w$ be the unique fat vertex of $\n^0$, and let $u\in V_{s}(\n^{0})$ and $v\in V_{s}(\n^{1})$.
Since $N_{\n}^{f}(u)=N_{\n_{0}}^{f}(u)=\{w\}$, it follows from Definition~\ref{definition:sum}~(\ref{s5}) that
\begin{align*}
u \sim v \text{ in } \n &\iff |N_{\n}^f(u) \cap N_{\n}^f(v)|=1\\
&\iff |\{w\} \cap N_{\n}^f(v)|=1\\
&\iff w \sim v \text{ in } \n.
\end{align*}
This implies that
\begin{align*}
N_{\n}^s(u) \setminus V_s(\n^0) = N_{\n}^s(w) \setminus V_s(\n^0). 
\end{align*}
Consequently, $G$ contains a complete bipartite graph $H$ whose partite sets are $V_s(\n^0)$ and $N_{\n}^s(w) \setminus V_s(\n^0)$ as a (not necessarily induced) subgraph.
Since $|V_s(\n^0)|\geq 2$, $G-x$ is a connected for every $x \in V_s(\n^0)$.
This leads to the desired conclusion.
\end{proof}

\begin{lem}
\label{prop:disconn}
Let $\fH \subset \fO$ be a family with $\fH=\bar{\fH}$.
Let $\h\in \fH$ be a Hoffman graph having $n$ slim vertices, and suppose that there exists an integer $N \in \{5,6,\ldots,2n+1\}$ such that every slim $\fH$-line graph of order $N$ has a unique strict $\fH$-cover up to equivalence.
Then the slim subgraph of $\h$ is connected.
\end{lem}
\begin{proof}
By way of contradiction, suppose that the slim subgraph of $\h$ is disconnected.
Then the slim subgraph of $\h$ is isomorphic to $H \sqcup H'$ for two non-empty graphs $H$ and $H'$.
Note that $2\leq \lfloor \frac{N-1}{2}\rfloor \leq \lceil \frac{N-1}{2} \rceil \leq n=|V(H)|+|V(H')|$.
Hence there exist two non-empty subgraphs $A$ and $B$ of $H$ and two non-empty subgraphs $A'$ and $B'$ of $H'$ such that $|V(A)|+|V(A')|=\lfloor \frac{N-1}{2} \rfloor $ and $|V(B)|+|V(B')|=\lceil \frac{N-1}{2} \rceil $ (here $A$ might intersect with $B$ and $A'$ might intersect with $B'$).
Then $|V(A)|+|V(B)|+|V(A')|+|V(B')|=N-1$.
Let $A_0$, $B_0$, $A'_0$ and $B'_0$ be vertex-disjoint copies of $A$, $B$, $A'$ and $B'$, respectively.
We define the graph $G$ from $A_0\sqcup B_0\sqcup A'_0\sqcup B'_0$ by adding a new vertex $x$ and joining $x$ to all vertices of $A_0 \sqcup B_0 \sqcup A'_0 \sqcup B'_0$ (see the left graph in Figure~\ref{fig:wa}).
Note that $|V(G)|=N$.

We define the Hoffman graphs $\n$ (resp.\ $\m$) from $G$ by adding two fat vertices $z$ and $w$ such that $N_{\n}(z)=V(A_0) \sqcup V(A_0') \sqcup \{x\}$ and $N_{\n}(w)=V(B_0) \sqcup V(B_0')\sqcup \{x\}$ (resp.\ $N_{\m}(z)=V(A_0) \sqcup V(B_0')\sqcup \{x\}$ and $N_{\m}(w)=V(B_0) \sqcup V(A_0')\sqcup \{x\}$) (see the central graph and the right graph in Figure~\ref{fig:wa}).
Then we can represent $\n$ and $\m$ as
\begin{align*}
\n &= \ind{A_0 \sqcup A_0' \sqcup \{z\}}{\n} \oplus \ind{\{x,z,w\}}{\n} \oplus \ind{B_0 \sqcup B_0' \sqcup \{w\}}{\n} \mbox{ and}\\
\m &= \ind{A_0 \sqcup B_0' \sqcup \{z\}}{\m} \oplus \ind{\{x,z,w\}}{\m} \oplus \ind{B_0 \sqcup A_0' \sqcup \{w\}}{\m}.
\end{align*}
Since every addend in above sums is isomorphic to $\h_{2}$ or a subgraph of $\h$, both $\n$ and $\m$ are $\fH$-line Hoffman graphs.
Considering the fact that $\fH=\bar{\fH}$, this implies that $\n$ and $\m$ are strict $\fH$-covers of $G$.
Since $|V(G)|=N$, it follows from the assumption of the lemma that $G$ has a unique strict $\fH$-cover, i.e., there exists an isomorphism $\psi : \n \to \m$ such that $\psi|_G = \id_{V(G)}$.
Then
\begin{align*}
A_0 \sqcup A_0' \sqcup \{x\} &= N_{\n}^s(z) = \psi(N_{\n}^s(z))\\
&\in \{N_{\m}^s(z),N_{\m}^s(w)\}=\{A_0 \sqcup B_0' \sqcup \{x\},B_0 \sqcup A_0' \sqcup \{x\}\},
\end{align*}
which contradicts the fact that $A_{0}$ and $B_{0}$ are non-empty.
		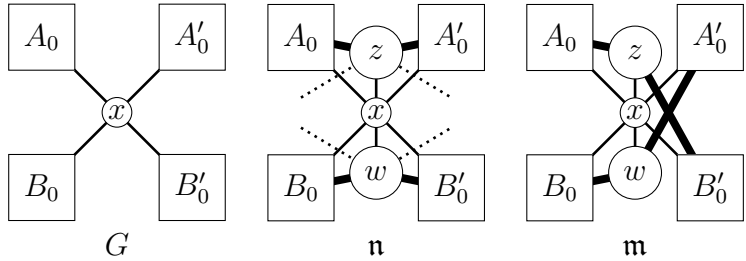
\begin{figure}[htbp]
		\centering
		\begin{tikzpicture}[
			gr/.style={rectangle,draw,fill=white,minimum size=25pt}]
			\node at (0,0) (d) [gr] {$B_0$};
			\node at (0,2) (a) [gr] {$A_0$};
			\node at (2,0) (c) [gr] {$B_0'$};
			\node at (2,2) (b) [gr] {$A_0'$};
			\node at (1,1) (s) [slim_char] {$x$};
			\node at (1,-0.76) [] {$G$};
			\begin{scope}[on background layer]
				\draw[line width = 1pt] (a.center)--(s.center)--(b.center) (c.center)--(s.center)--(d.center);
			\end{scope}
		\end{tikzpicture}
		\quad
		\begin{tikzpicture}[
			gr/.style={rectangle,draw,fill=white,minimum size=25pt}]
			\node at (0,0) (d) [gr] {$B_0$};
			\node at (0,2) (a) [gr] {$A_0$};
			\node at (2,0) (c) [gr] {$B_0'$};
			\node at (2,2) (b) [gr] {$A_0'$};
			\node at (1,1) (s) [slim_char] {$x$};
			\node at (1,1.8) (f) [fat_char] {$z$};
			\node at (1,0.2) (g) [fat_char] {$w$};
			\node at (1,-0.8) [] {$\n$};
			\begin{scope}[on background layer]
				\draw[line width = 1pt] (a.center)--(s.center)--(b.center) (c.center)--(s.center)--(d.center);
				\draw[line width = 1pt] (f.center)--(s.center)--(g.center);
				\draw[line width = 3pt] (d.center)--(g.center)--(c.center);
				\draw[line width = 3pt] (a.center)--(f.center)--(b.center);
				\draw[line width = 1pt,dotted] (0,1.2)--(f)--(2,1.2);
				\draw[line width = 1pt,dotted] (0,0.8)--(g)--(2,0.8);
			\end{scope}
		\end{tikzpicture}
		\quad
		\begin{tikzpicture}[
			gr/.style={rectangle,draw,fill=white,minimum size=25pt}]
			\node at (0,0) (d) [gr] {$B_0$};
			\node at (0,2) (a) [gr] {$A_0$};
			\node at (2,0) (c) [gr] {$B_0'$};
			\node at (2,2) (b) [gr] {$A_0'$};
			\node at (1,1) (s) [slim_char] {$x$};
			\node at (1,1.8) (f) [fat_char] {$z$};
			\node at (1,0.2) (g) [fat_char] {$w$};
			\node at (1,-0.8) [] {$\m$};				
			\begin{scope}[on background layer]
				\draw[line width = 1pt] (a.center)--(s.center)--(b.center) (c.center)--(s.center)--(d.center);
				\draw[line width = 1pt] (f.center)--(s.center)--(g.center);
				\draw[line width = 3pt] (d.center)--(g.center)--(b.center);
				\draw[line width = 3pt] (a.center)--(f.center)--(c.center);
			\end{scope}
		\end{tikzpicture}
		\caption{The graphs in the proof of Lemma~\ref{prop:disconn}}
		\label{fig:wa}
	\end{figure}
\end{proof}

\begin{lem}
\label{lem:apple}
Let $\h\in \fO \setminus \{\h_2\}$, and let $G$ be a non-empty subgraph of the slim subgraph of $\h$.
Then $\bar{G}$ is connected if and only if $\hind{V(G)}{\h}$ is indecomposable.
\end{lem}
\begin{proof}
Note that $\bar{G}$ is disconnected if and only if there exists a partition $\{A,B\}$ of $V(G)$ such that
\begin{align}
\mbox{$x \sim y$ in $G$ for all $x \in A$ and $y \in B$.}\label{cond-lem:apple}
\end{align}
Thus it suffices to show that there exists a partition $\{A,B\}$ of $V(G)$ satisfying (\ref{cond-lem:apple}) if and only if $\hind{V(G)}{\h}$ is decomposable.
Recall that $\h$ has exactly one fat vertex.
Since a partition $\{A,B\}$ of $V(G)$ with (\ref{cond-lem:apple}) satisfies $\hind{V(G)}{\h} = \hind{A}{\h} \oplus \hind{B}{\h}$, we obtain the ``only if'' part; if $\hind{V(G)}{\h}$ is decomposable, i.e., there exist non-empty Hoffman graphs $\n$ and $\m$ with $\hind{V(G)}{\h} = \n \oplus \m$, then $A := V_s(\n)$ and $B := V_s(\m)$ satisfies (\ref{cond-lem:apple}), and hence we obtain the ``if'' part.
\end{proof}

\begin{lem}
\label{lem:conn2}
Let $\fH \subset \fO$ be a family with $\fH=\bar{\fH}$, and let $\h\in \fH$ be a Hoffman graph with $n$ slim vertices.
Suppose that there exists an integer $N\geq 7$ such that every slim $\fH$-line graph of order $N$ has a unique strict $\fH$-cover up to equivalence.
Then the following hold:
\begin{enumerate}
\item \label{c1}
If $n \ge (N-1)/2$, then the slim subgraph of $\h$ is connected.
\item \label{c2}
If $n \geq 3$, then there are two distinct slim vertices $\alpha$ and $\beta$ of $\h$ such that all of $\h-\alpha$, $\h-\beta$ and $\h-\{\alpha,\beta\}$ are indecomposable.
\item \label{c3}
If $n \geq N-2$, then there are two distinct slim vertices $\alpha$ and $\beta$ such that all of $\h-\alpha$, $\h-\beta$ and $\h-\{\alpha,\beta\}$ are indecomposable and their slim subgraphs are connected.
\end{enumerate}
\end{lem}
\begin{proof}
If $n \geq (N-1)/2$, i.e., $N \in \{7,8,\ldots,2n+1\}$, then the slim subgraph of $\h$ is connected by Lemma~\ref{prop:disconn}, which proves that (\ref{c1}) holds.

We assume $n \geq 3$, and prove that (\ref{c2}) holds.
Note that $\h$ is not isomorphic to $\h_2$. 
Let $G$ be the slim subgraph of $\h$.
Since $\h$ is indecomposable, $\bar{G}$ is connected by Lemma~\ref{lem:apple}.
Hence $\bar{G}$ has a spanning tree $T$.
Since $|V(T)|=|V(\bar{G})|=n\geq 3$, there exist two vertices $\alpha$ and $\beta$ of $T$ such that $T-\alpha$, $T-\beta$ and $T-\{\alpha,\beta\}$ are connected.
Since $\hind{V(\overline{T-X})}{\h}=\h-X$ for any $X\subset V(T)$, it follows from Lemma~\ref{lem:apple} that all of $\h-\alpha$, $\h-\beta$ and $\h-\{\alpha,\beta\}$ are indecomposable, which proves that (\ref{c2}) holds.

To prove that (\ref{c3}) holds, we consider the case $n \geq N-2$.
Since $n \geq 3$, we can take two slim vertices $\alpha$ and $\beta$ of $\h$ satisfying the condition in (\ref{c2}).
Fix $X\in \{\{\alpha \},\{\beta \},\{\alpha ,\beta \}\}$.
It suffices to show that the slim subgraph of $\h-X$ is connected.
Since $\h\in \fO$ and $\h-X$ is indecomposable, we have $\h-X\in \fO$.
Furthermore, 
\begin{align*}
|V_s(\h-X)|\geq  n-2 \geq N-4 \geq \frac{N-1}{2}
\end{align*}
since $N\geq 7$.
Hence by applying (\ref{c1}) to $\h-X$, the slim subgraph of $\h-X$ is connected.
This completes the proof of Lemma~\ref{lem:conn2}
\end{proof}

\begin{rem}
In Lemma~\ref{lem:conn2}~\eqref{c2} and~\eqref{c3}, we can find two slim vertices $\alpha$ and $\beta$ assuring us that $\h-\{\alpha,\beta\}$ has good properties. One might notice that the above fact is not used in this paper. 
Actually it will be used in our following paper and gives almost no influence to the shortness of the proof. Thus we give it in the lemma.
\end{rem}
\subsection{Order of Hoffman graphs}\label{sec:order}

\begin{dfn}[\bf{Family $\fH(m)$}]
\label{dfn:order}
Let $\fH$ be a non-empty subfamily of $\fO$.
In the remaining of this paper, we fix an order $\g_0,\g_1,\ldots $ of the elements of $\bar{\fH}$ so that
$$
|V_s(\g_0)|\leq |V_s(\g_1)|\leq \cdots.
$$
For each $m \in \mathbb{N}\cup \{0\}$, let $\fH(m):=\{\g_i :0\leq i\leq \min\{m, |\bar{\fH}|-1\}\}$.
\end{dfn}

\begin{lem}
\label{rem:LG}
Let $\fH \subset \fO$ be a non-empty family, and let $\g_{0},\g_{1},\ldots $ be as in Definition~\ref{dfn:order}.
Then the following hold:
\begin{enumerate}
\item
We have $\g_{0}\simeq \h_{2}$.
In particular, $\fH(0)=\{\h_{2}\}$.
\item
If $|\fH|\geq 2$, then $\g_{1}\simeq \h_{3}$.
In particular, $\fH(1)=\{\h_{2},\h_{3}\}$.
\item
If $|\fH|\geq 3$, then $|V_{s}(\g_{i})|\geq 3$ for all $i=2,3, \ldots $.
\end{enumerate}
\end{lem}
\begin{proof}
By the definition of $\bar{\fH}$, (1) clearly holds. 

Assume that $|\fH |\geq 2$.
Then there exists an indecomposable Hoffman graph $\g \in \fH \setminus \{ \h_2 \}$.
Note that $|V_{s}(\g)|\geq 2$ and $\hind{V_{s}(\g)}{\g}=\g$.
By applying Lemma~\ref{lem:apple} with $\h = \g$, the complement of $\ind{V_{s}(\g)}{\g}$ is connected, and hence $\g$ has two non-adjacent slim vertices $x$ and $y$.
It follows that $\g$ has a Hoffman subgraph $\hind{\{x,y\}}{\g}\simeq \h_{3}$.
On the other hand, since $\h_3$ is the unique Hoffman graph in $\fO$ with exactly two slim vertices, $\h_3 \in \bar{\fH}$.
Consequently, $\g_1 \simeq \h_3$, which proves (2).

As we mentioned above, $\h_{3}$ is the unique Hoffman graph in $\fO$ with exactly two slim vertices.
This together with (2) leads to (3).
\end{proof}

\begin{lem}
\label{lem:Hm}
For $\fH \subset \fO$ and $m \in \mathbb{N}\cup \{0\}$, we have $\fH(m)=\overline{\fH(m)}$.
\end{lem}
\begin{proof}
By the definition of $\overline{\fH(m)}$, we obtain $\fH(m) \subset \overline{\fH(m)}$.
Hence it suffices to show that $\g \in \fH(m)$ for every $\g\in \overline{\fH(m)}$.
Let $\g_{0},\g_{1},\ldots $ be as in Definition~\ref{dfn:order}.
If $\g \simeq \h_2$, then $\g \in \fH(m)$ since $\g_{0} \simeq \h_{2}$ and $\h_{2}\in \fH(m)$, 
Thus we may assume that $\g \not\simeq \h_2$.
Then there exists $\g_i\in \fH(m)$ such that $\g$ is a Hoffman subgraph of $\g_i$.
If $|V_s(\g)| = |V_s(\g_i)|$, then $\g = \g_i$, and so $\g \in \fH(m)$, as desired.
Thus we may assume that $|V_s(\g)| < |V_s(\g_i)|$.
Since $\g \in \fO$ and $\g$ is a subgraph of $\g_i \in \fH$, we have $\g \in \bar{\fH}$.
Hence there exists an index $j$ with $j<i~(\leq m)$ such that $\g \simeq \g_j$, and so $\g \in \fH(m)$.
\end{proof}

\begin{lem}
\label{lem:finite}
Let $\fH\subset \fO$ be a family with $\h_2 \in \fH$, and let $G$ be a slim $\fH$-line graph.
Then the following conditions are equivalent:
\begin{enumerate}
\item \label{z1}
$G$ has a unique strict $\bar{\fH}$-cover up to equivalence.
\item \label{z2}
For every $m \in \mathbb{N}\cup \{0\}$, if $G$ is a slim $\fH(m)$-line graph, then $G$ has a unique strict $\fH(m)$-cover up to equivalence.
\end{enumerate}
\end{lem}
\begin{proof}
We first show that (\ref{z1}) implies (\ref{z2}).
Suppose that the condition~(\ref{z1}) holds, i.e., $G$ has a unique strict $\bar{\fH}$-cover $\h$.
Let $m$ be a non-negative integer such that $G$ is a slim $\fH(m)$-line graph.
Then by applying Lemma~\ref{lem:ExiCov} with $\fH = \fH(m)$ and $\h$, $G$ has a strict $\overline{\fH(m)}$-cover $\g$.
This together with Lemma~\ref{lem:Hm} implies that $\g$ is a strict $\fH(m)$-cover of $G$, and hence $\g$ is also a strict $\bar{\fH}$-cover of $G$.
By the uniqueness of $\h$ for strict $\bar{\fH}$-covers of $G$, it follows that $\h$ and $\g$ are equivalent, which leads to (\ref{z2}).

Next we show that (\ref{z2}) implies (\ref{z1}).
Suppose that the condition~(\ref{z2}) holds.
Since $G$ is a finite slim $\fH$-line graph, there exists a sufficiently large integer $M$ such that
\begin{enumerate}
\item[(a)]
$G$ is a slim $\fH(M)$-line graph, and
\item[(b)]
every strict $\bar{\fH}$-cover of $G$ is a strict $\fH(M)$-cover.
\end{enumerate}
By (\ref{z2}), $G$ has a unique strict $\fH(M)$-cover $\h$ up to equivalence, and hence $\h$ is also a strict $\bar{\fH}$-cover of $G$.
By the uniqueness of $\h$ for strict $\fH(M)$-covers of $G$, it follows from (b) that $\h$ is a unique strict $\bar{\fH}$-cover of $G$, which leads to (\ref{z1}).
\end{proof}

\section{Proof of Theorem~\ref{thm:uni}}\label{sec:cov}

Recall that the following is our main result in this paper.
\vspace{2.3mm}

\fakethm{thm:uni}{
Let $\fH \subset \fO$ be a family with $\h_2 \in \fH$, and let $N\geq 7$ be an integer.
If every connected slim $\fH$-line graph of order $N$ has exactly one strict $\bar{\fH}$-cover up to equivalence, then every connected slim $\fH$-line graph of order at least $N$ has exactly one strict $\bar{\fH}$-cover up to equivalence.}

\begin{proof}[Proof of Theorem~\ref{thm:uni}]
Let $\fH$ and $N$ be as in Theorem~\ref{thm:uni}, and let $\g_{0},\g_{1},\ldots $ be as in Definition~\ref{dfn:order}.
Suppose that every slim $\fH$-line graph of order $N$ has exactly one strict $\bar{\fH}$-cover up to equivalence.
To prove the theorem, we show that every connected slim $\fH$-line graph $G$ of order at least $N$ has exactly one strict $\bar{\fH}$-cover up to equivalence.
Considering Lemma~\ref{lem:finite}, it suffices to show that the following holds:
\begin{align}
&\mbox{For integers $m\geq 0$ and $n\geq N$, every connected slim $\fH(m)$-line graph $G$ of order $n$}\nonumber \\
&\mbox{has a unique strict $\fH(m)$-cover.}\label{main-condition}
\end{align}
We prove (\ref{main-condition}) by induction on $m+n$.
If $n=N$, then the assumption of the theorem implies (\ref{main-condition}); if $m\in \{0,1\}$, then (\ref{main-condition}) holds by Theorem~\ref{thm:GLG}, Corollary~\ref{rem:GLG} and Lemma~\ref{rem:LG} (1) and (2).
Thus we may assume $n\geq N+1$ and $m\geq 2$.

Let $G$ be a connected slim $\fH(m)$-line graph of order $n$.
If every strict $\fH(m)$-cover of $G$ is a strict $\fH(m-1)$-cover, then (\ref{main-condition}) holds by the induction hypothesis.
Thus we may assume that $G$ has a strict $\fH(m)$-cover $\n=\bigoplus_{i=0}^k \n^i$ such that $\n^0 \simeq \g_{m}$ and $\n^i \in \fH(m)$ for all $i=1, \ldots, k$.
By Lemma~\ref{rem:LG} (3), we have $|V_{s}(\n^{0})|\geq 3$.
Let $\m$ be a strict $\fH(m)$-cover of $G$.
The equivalency of $\n$ and $\m$ implies (\ref{main-condition}).
Hence our final goal in this proof is to prove that
\begin{align}
\mbox{there exists an isomorphism $\Phi :\n \to \m$ such that $\Phi|_G = \id_{V(G)}$.}\label{main-condition2}
\end{align}

\begin{clm}
\label{clm:banana}
There exist two vertices $\alpha,\beta\in V_s(\n^0)$ such that
\begin{enumerate}
\item[{\rm (a)}]
both $G-\alpha$ and $G-\beta$ are connected; and
\item[{\rm (b)}]
both $\n^0 - \alpha$ and $\n^0 - \beta$ are indecomposable.
\end{enumerate}
\end{clm}
\begin{proof}[Proof of Claim~\ref{clm:banana}]
If $k=0$ (i.e., $\n=\n^{0}$), then the claim holds by Lemma~\ref{lem:conn2}~(\ref{c3}) with $\fH=\fH(m)$ and $\h=\n^{0}$.
Thus we may assume that $k\geq 1$.
Since $|V_s(\n^0)| \geq 3$, it follows from Lemma~\ref{lem:conn2}~(\ref{c2}) that there exist vertices $\alpha,\beta\in V_s(\n^0)$ satisfying (b).
Since $\n = \n^{0} \oplus (\bigoplus _{i=1}^{k}\n^{i})$, the vertices $\alpha $ and $\beta $ satisfy (a) by Lemma~\ref{lem:conn}.
\end{proof}

Let $\alpha $ and $\beta $ be slim vertices of $\n^{0}$ as in Claim~\ref{clm:banana}.
Applying Lemma~\ref{lem:tani} with $\h=\n$, $\h^{i}=\n^{i}$ and $X=V_s(\n) \setminus \{\alpha\}$, we have
$$
\hind{ V_s(\n)\setminus \{\alpha\}}{\n} = \bigoplus_{i=0}^k \hind{V_s(\n^i)\setminus \{\alpha\}}{\n^i}.
$$
This together with Lemma~\ref{lem:tilde} implies that
\begin{align}
\tildehind{V_s(\n)\setminus \{\alpha\}}{\n}=\bigoplus_{i=0}^k \tildehind{V_s(\n^i)\setminus \{\alpha\}}{\n^i}=(\n^0-\alpha) \oplus \left(\bigoplus_{i=1}^k \n^i \right).\label{d1}
\end{align}
By the symmetry of $\alpha $ and $\beta $, we have
\begin{align}
\tildehind{V_s(\n)\setminus \{\beta\}}{\n} = \bigoplus_{i=0}^k \tildehind{V_s(\n^i)\setminus \{\beta\}}{\n^i}=(\n^0-\beta) \oplus \left(\bigoplus_{i=1}^k \n^i \right).\label{d'1}
\end{align}

Recall that $\m$ is a strict $\fH(m)$-cover of $G$.
Now we write $\m=\bigoplus_{i=0}^l\m^i$ where $\alpha \in V_s(\m^0)$ and $\m^i \in \fH(m)$ for $i=0,\ldots,l$.
By similar argument for (\ref{d1}), we obtain that
\begin{align}
\tildehind{V_s(\m)\setminus \{\alpha\}}{\m}=\bigoplus_{i=0}^l \tildehind{V_s(\m^i)\setminus \{\alpha\}}{\m^i}= \tildehind{V_s(\m^0)\setminus \{\alpha\}}{\m^0} \oplus \left(\bigoplus_{i=1}^l \m^i \right)\label{d2}
\end{align}
and
\begin{align}
\tildehind{V_s(\m)\setminus \{\beta \}}{\m}=\bigoplus_{i=0}^l \tildehind{V_s(\m^i)\setminus \{\beta \}}{\m^i}.\label{d'2}
\end{align}

Note that for $x\in \{\alpha , \beta \}$, the slim subgraphs of $\tildehind{V_s(\n)\setminus \{x\}}{\n}$ and $\tildehind{V_s(\m)\setminus \{x\}}{\m}$ are equal to $G-x$.
Since every addend appearing in (\ref{d1})--(\ref{d'2}) belongs to $\fH(m)$, it follows from (\ref{d1})--(\ref{d'2}) that
\begin{enumerate}
\item[$\bullet $]
$\tildehind{V_s(\n) \setminus \{\alpha\}}{\n}$ and $\tildehind{V_s(\m) \setminus \{\alpha\}}{\m}$ are strict $\fH(m)$-cover of $G-\alpha$; and
\item[$\bullet $]
$\tildehind{V_s(\n) \setminus \{\beta \}}{\n}$ and $\tildehind{V_s(\m) \setminus \{\beta \}}{\m}$ are strict $\fH(m)$-cover of $G-\beta $.
\end{enumerate}
Since $G-\alpha $ and $G-\beta $ are connected, this together with the induction hypothesis implies that there are isomorphisms
$$
\varphi : (\n^0-\alpha)\oplus \left(\bigoplus_{i=1}^k \n^i \right) \to \tildehind{V_s(\m^0)\setminus \{\alpha\}}{\m^0} \oplus \left(\bigoplus_{i=1}^l \m^i \right)
$$
and
$$
\psi : (\n^0-\beta) \oplus \left(\bigoplus_{i=1}^k \n^i \right) \to \bigoplus_{i=0}^l \tildehind{V_s(\m^i)\setminus \{\beta\}}{\m^i}
$$
such that $\varphi|_{G-\alpha}=\id_{V(G)\setminus \{\alpha \}}$ and $\psi|_{G-\beta}=\id_{V(G)\setminus \{\beta \}}$.

\begin{clm}
\label{clm:nasi}
We have $\m^0-\alpha =\hind{V_s(\m^0) \setminus \{\alpha\}}{\m^0}= \varphi(\n^0-\alpha) \in \fH \setminus \{ \h_2\}$.
\end{clm}
\begin{proof}[Proof of Claim~\ref{clm:nasi}]
Recall that $|V_{s}(\n^{0})|\geq 3$.
Let $\gamma \in V_s(\n^0)\setminus \{\alpha,\beta\}$.
We first prove that
\begin{align}
\gamma  \in V_s(\m^0).\label{cond-clm:nasi}
\end{align}
Since $\alpha \in V(\n^{0})$ and $\psi|_{G-\beta}=\id_{V(G)\setminus \{\beta \}}$, we have $\alpha \in V(\psi(\n^0-\beta))$.
Hence $\alpha $ is a common vertex of $\psi(\n^0-\beta)$ and $\m^{0}$.
In particular, $V_{s}(\psi(\n^0-\beta))\cap V_{s}(\m^{0})\neq \emptyset $.
On the other hand, since $\n^0-\beta$ is indecomposable, it follows from Corollary~\ref{cor:irr} that $\psi(\n^0-\beta)$ is also indecomposable.
Consequently, $\psi(\n^0-\beta)$ is an indecomposable addend of $\bigoplus_{i=0}^l \tildehind{V_s(\m^i)\setminus \{\beta\}}{\m^i}$ intersecting with $V_{s}(\m^{0})$.
This implies that $\psi(\n^0-\beta)$ is an indecomposable addend of $\tildehind{V_s(\m^0)\setminus \{\beta\}}{\m^0}$.
Since $\psi|_{G-\beta}=\id_{V(G)\setminus \{\beta\}}$ and $\gamma $ is a vertex of $\n^{0}-\beta$, this implies that $\gamma $ is a vertex of $\m^{0}$, which proves (\ref{cond-clm:nasi}).

Since $\gamma \in V(\n^{0})$ and $\varphi|_{G-\alpha}=\id_{V(G)\setminus \{\alpha \}}$, we have $\gamma \in V(\varphi (\n^{0}-\alpha ))$.
This together with (\ref{cond-clm:nasi}) implies that $\gamma$ is a common slim vertex of $\varphi (\n^{0}-\alpha )$ and $\m^0$.
In particular, $V_{s}(\varphi (\n^0-\alpha ))\cap V_{s}(\m^{0})\neq \emptyset $.
On the other hand, since $\n^0-\alpha $ is indecomposable, it follows from Corollary~\ref{cor:irr} that $\varphi (\n^0-\alpha )$ is also indecomposable.
Consequently, $\varphi (\n^0-\alpha )$ is an indecomposable addend of $\bigoplus_{i=0}^k \tildehind{V_s(\m^i)\setminus \{\alpha \}}{\m^i}$ intersecting with $V_{s}(\m^{0})$.
This implies that $\varphi (\n^0-\alpha )$ is an indecomposable addend of $\tildehind{V_s(\m^0)\setminus \{\alpha \}}{\m^0}$, and hence $V_{s}(\varphi (\n^{0}-\alpha ))\subset V_s(\m^0)\setminus \{\alpha\}$.
Since $\m^0, \n^0 \in \fH(m)$, $\n^0 \simeq \g_m$ and $|V_{s}(\g_{m})|=\max\{|V_{s}(\g)|:\g\in \fH(m)\}$, we have
$$
|V_s(\g_m)|-1 \geq |V_s(\m^0)\setminus \{\alpha\}|\geq |V_{s}(\varphi (\n^{0}-\alpha ))|=|V_s(\n^0-\alpha)| = |V_s(\g_m)|-1.
$$
This chain of inequalities forces $\varphi(\n^0-\alpha)=\tildehind{V_s(\m^0) \setminus \{\alpha\}}{\m^0}$.
Since $|V_{s}(\n^{0})|\geq 3$, this implies that $\tildehind{V_s(\m^0) \setminus \{\alpha\}}{\m^0}=\m^{0}-\alpha $.
Consequently, we have $\m^0-\alpha=\varphi(\n^0-\alpha)$.

Recall that $\n^{0}\in \fH(m)$, $|V_{s}(\n^{0}-\alpha )|\geq 2$ and $\n_{0}-\alpha $ is indecomposable.
Hence $\n^0 - \alpha\in \fH \setminus \{\h_2\}$.
This completes the proof of the claim.
\end{proof}

Let $\Phi :V(\n) \to V(\m)$ be the mapping such that
$$
\Phi (x)=
\begin{cases}
\alpha & (x=\alpha )\\
\varphi (x) & (x\neq \alpha ).
\end{cases}
$$
Now we show that $\Phi $ satisfies (\ref{main-condition2}).

Since $\varphi $ is an isomorphism from $\n-\alpha $ to $\m-\alpha $ and $\Phi (\alpha )=\alpha $, $\Phi $ is a bijection from $V(\n)$ to $V(\m)$ such that $\Phi(V_{s}(\n_{0}))=V_{s}(\m_{0})$ and $\Phi(V_{f}(\n_{0}))=V_{f}(\m_{0})$.
Since $\varphi |_{V_{s}(\h)\setminus \{\alpha \}}=\varphi |_{G-\alpha }=\id_{V(G)\setminus \{\alpha \}}$ and $\Phi (\alpha )=\alpha $, we have
\begin{align}
\Phi |_{G}=\id_{V(G)}.\label{clm:saigo:0}
\end{align}
Since $\varphi $ is an isomorphism from $\n-\alpha $ to $\m-\alpha $, it follows from Claim~\ref{clm:nasi} that $|V_{f}(\n^{0})|=|V_{f}(\m^{0})|=1$ and $V_{f}(\m^{0})=\{\Phi (w)\}$ where $w$ is the unique fat vertex of $\n^{0}$.
Since $\alpha $ is a common slim vertex of $\n^{0}$ and $\m^{0}$, $N_{\n}^{f}(\alpha )=N_{\n^{0}}^{f}(\alpha )=\{w\}$ and $N_{\m}^{f}(\alpha )=N_{\m^{0}}^{f}(\alpha )=\{\Phi (w)\}$, i.e.,
\begin{align*}
\mbox{for $x\in V_{f}(\n) $, $\alpha \sim x$ in $\n $ if and only if $\Phi(\alpha ) \sim \Phi(x)$ in $\m$.}
\end{align*}
Since $\varphi $ is an isomorphism from $\n-\alpha $ to $\m-\alpha $ and $\Phi |_{G}=\id_{V(G)}$ by (\ref{clm:saigo:0}), this implies that
\begin{align*}
\mbox{for $x,y\in V(\n) $, $x \sim y$ in $\n $ if and only if $\Phi(x) \sim \Phi(y)$ in $\m$.}
\end{align*}
Hence $\Phi $ is an isomorphism from $\n$ to $\m$.
This together with (\ref{clm:saigo:0}) leads to (\ref{main-condition2}).
\end{proof}

\section{Computer approach}	\label{sec:computer}

\begin{dfn}[{\bf Integer $N_{\fH}$}]
For a family $\fH \subset \fO$ with $\h_2 \in \fH$, if there exists an integer $N\geq 7$ satisfying that every connected slim $\fH$-line graph of order $N$ has exactly one strict $\bar{\fH}$-cover up to equivalence, then let $N_{\fH}$ be the smallest integer at least $7$ satisfying it; otherwise, let $N_{\fH}=\infty$.
\end{dfn}

Theorem~\ref{thm:uni} asserts that, for a family $\fH \subset \fO$ with $\h_2 \in \fH$, every connected slim $\fH$-line graphs of order at least $N_{\fH}$ has a strict $\bar{\fH}$-cover.
In this section, we explain a way with computer to verify whether every connected slim $\fH$-line graph of order $N$ has a unique strict $\fH$-cover for an integer $N\geq 7$ and a family $\fH$ of Hoffman graphs.
We start with a lower bound of $N_{\fH}$.

\begin{prop}
\label{prop:lb}
	Let $\fH \subset \fO$ be a family with $\h_2 \in \fH$.
	Then $N_{\fH}\geq 2|V_s(\h)|+2$ for every $\h \in \bar{\fH}$ whose slim subgraph is disconnected.
\end{prop}
\begin{proof}
	Take a Hoffman graph $\h \in \fH$ such that $|V_{s}(\h)|=n$ and the slim subgraph of $\h$ is disconnected.
	Then there is no integer $N$ satisfying the assumption of Lemma~\ref{prop:disconn}.
	This forces $N_{\fH}\geq 2n+2$.
\end{proof}

For a Hoffman graph $\h$, we let $\Aut(\h)$ denote the automorphism group of $\h$, and let
\begin{align*}
	\Aut^*(\h) := \{ \psi \in \Aut(\h) : \psi|_{V_s(\h)} = \id_{V_{s}(\h)} \}.
\end{align*}
For a Hoffman graph $\h$ whose indecomposable decomposition is $\bigoplus_{i \in I} \h^i$, let
\[
	\mathcal{A}(\h) := \left\{ (\psi^i)_{i \in I} \in \prod_{i \in I}\Aut^*(\h^i) : \forall i,j \in I\mbox{ with }i\neq j, \forall x \in V_f(\h^i) \cap V_f(\h^j), \psi^i(x)=\psi^j(x) \right\}.
\]

We will prove the following two propositions which give useful properties for our strategy using computer search.

\begin{prop}
\label{lem:Aut}
	Let $\h = \bigoplus_{i\in I}\h^i$ be a connected Hoffman graph with $\h^i \in \fO$ for every $i\in I$.
	If $\h$ is not isomorphic to $\h_2$, then $\Aut^*(\h) = \{\id_{V(\h)}\}$.
\end{prop}

\begin{prop}
\label{lem:pc}
Let $\fH$ be a family of Hoffman graphs, and $N$ be a positive integer.
Let $\cX$ be the family of Hoffman graphs $\bigoplus_{i\in I} \h^i$ up to isomorphism such that
\begin{enumerate}
\item[{\rm (a)}]
$\h^i \in \fH$ for every $i\in I$;
\item[{\rm (b)}]
$|V_{s}(\h)|=N$; and
\item[{\rm (c)}]
the slim subgraph of $\h$ is connected.
\end{enumerate}
Let $\cY$ be the family of connected slim $\fH$-line graphs of order $N$ up to isomorphism.
Define the mapping $\Phi :\cX\to \cY$ so that $\Phi (\h)$ is the slim subgraph of $\h$.
If
\begin{enumerate}
\item[{\rm (I)}]
$\Phi $ is surjective;
\item[{\rm (II)}]
$\Aut^*(\n) = \{\id_{V(\n)}\}$ for every $\n \in \cX$;
\item[{\rm (III)}]
$|\cX| = |\cY|$; and
\item[{\rm (IV)}]
$|\Aut(\n)| = |\Aut(\Phi (\n))|$ for every $\n \in \cX$,
\end{enumerate}
then every graph in $\cY$ has a unique strict $\fH$-cover up to equivalence.
\end{prop}

\begin{lem}
\label{lem:fixfat}
Let $\h=\n\oplus \m$ be a Hoffman subgraph, and suppose that $\n$ and $\m$ are non-empty.
Then $\psi (x)=x$ for all $x \in V_f(\n) \cap V_f(\m)$ and $\psi \in \Aut^{*}(\h)$.
\end{lem}
\begin{proof}
Let $x \in V_f(\n) \cap V_f(\m)$ and $\psi \in \Aut^{*}(\h)$.
By Definition~\ref{dfn:Hoff} (1), there exist vertices $y \in V_s(\n)$ and $z \in V_s(\m)$ such that $x\sim y$ in $\n$ and $x\sim z$ in $\m$.
In particular, $|N_{\h}^{f}(y)\cap N_{\h}^{f}(z)|=1$, and so $|N_{\h}^{f}(\psi (y))\cap N_{\h}^{f}(\psi (z))|=1$.
This leads to $N_{\h}^{f}(\psi (y))\cap N_{\h}^{f}(\psi (z))=\{\psi(x)\}$.
Since $\psi (y)=y$ and $\psi (z)=z$, we have
$$
\{x\}=N_{\h}^{f}(y)\cap N_{\h}^{f}(z)=N_{\h}^{f}(\psi (y))\cap N_{\h}^{f}(\psi (z))=\{\psi(x)\},
$$
as desired.
\end{proof}

\begin{lem}
\label{lem:decomp}
Let $\h=\bigoplus_{i \in I} \h^i$ be a Hoffman graph, and suppose that $\h^i$ is non-empty and indecomposable for every $i \in I$.
Then the following hold:
\begin{enumerate}
\item
For $\psi \in \Aut^*(\h)$, we have $(\psi|_{\h^i})_{i \in I} \in \mathcal{A}(\h)$.
\item
For $(\psi^i)_{i \in I} \in \mathcal{A}(\h)$, we define the mapping $\psi : V(\h) \to V(\h)$ by
\begin{equation}
\psi(x) = \psi^i(x) \quad \text{if $x \in V(\h^i)$}.\label{K10}
\end{equation}
Then we have $\psi \in \Aut^*(\h)$.
\item
The mapping $\varphi : \Aut^*(\h) \to \mathcal{A}(\h)$ defined by $\varphi(\psi) = (\psi|_{\h^i})_{i \in I}$ is bijective.
\end{enumerate}
\end{lem}
\begin{proof}
We first prove (1).
Now we show that
\begin{align}
\mbox{$\psi(x) \in V(\h^i)$ for every $x \in V(\h^i)$.}\label{cond-lem:decomp-1}
\end{align}
If $x\in V_{s}(\h^{i})$, then $\psi (x)=x\in V(\h^{i})$ since $\psi |_{V_{s}(\h)}=\id_{V_{s}(\h)}$.
Thus we may assume that $x\in V_{f}(\h^{i})$.
Then there exists a vertex $u\in N_{\h^{i}}^{s}(x)$.
Since $\psi \in \Aut^*(\h)$, we have $u = \psi(u) \sim \psi(x)$ in $\h$, i.e., $\psi (x)\in N_{\h}^{f}(u)$.
It follows from Definition~\ref{definition:sum}~(\ref{s3}) that $\psi(x) \in N_{\h^i}^f(u)$, and so $\psi(x) \in V(\h^i)$, which proves (\ref{cond-lem:decomp-1}).
Since $\psi \in \Aut^{*}(\h)$ and $\psi |_{\h^{i}}:V(\h^{i})\to V(\h^{i})$ by (\ref{cond-lem:decomp-1}), $\psi |_{\h^{i}}\in \Aut^{*}(\h^{i})$.
Furthermore, for $i,j\in I$ with $i\neq j$ and $y\in V_{f}(\h^{i})\cap V_{f}(\h^{j})$, it follows from Lemma~\ref{lem:fixfat} that $\psi (y)=y$, and so $\psi |_{\h^{i}}(y)=\psi |_{\h^{j}}(y)$.
Consequently, we have $(\psi|_{\h^i})_{i \in I} \in \mathcal{A}(\h)$.

Next we prove (2).
Since $\psi ^{i}|_{V_{s}(\h^{i})}=\id_{V(\h^{i})}$ for all $i\in I$, we have $\psi |_{V_{s}(\h)}=\id_{V_{s}(\h)}$.
Thus it suffices to show that $\psi \in \Aut(\h)$, i.e., $x\sim y$ in $\h$ if and only if $\psi (x)\sim \psi (y)$ in $\h$ for all $x,y\in V(\h)$ with $x\neq y$.
Let $x,y\in V(\h)$ be vertices with $x\neq y$.
If $x,y\in V_{s}(\h)$, then $x\sim y$ in $\h$ if and only if $\psi (x)\sim \psi (y)$ in $\h$ since $\psi |_{V_{s}(\h)}=\id _{V_{s}(\h)}$; if $x,y\in V_{f}(\h)$, then $x\not\sim y$ and $\psi (x)\not\sim \psi (y)$ in $\h$ by Definition~\ref{dfn:Hoff} (2) and the fact that $\psi (V_{f}(\h))\subset V_{f}(\h)$.
In either case, we obtain the desired conclusion.
Thus, without loss of generality, we may assume that $x\in V_{s}(\h^{i})$ with $i\in I$ and $y\in V_{f}(\h)$.
If $y\notin V_{f}(\h^{i})$, then $\psi (y)\notin V_{f}(\h^{i})$, and hence $x \not\sim y$ and $\psi(x) \not\sim \psi(y)$ in $\h$ by Definition~\ref{definition:sum}~(\ref{s3}); if $y\in V_{f}(\h^{i})$, then 
\begin{align*}
x\sim y\mbox{ in } \h &\iff x\sim y\mbox{ in } \h^{i}\\
&\iff \psi ^{i}(x) \sim \psi ^{i}(y) \mbox{ in }\h^{i}\\
&\iff \psi (x) \sim \psi (y) \mbox{ in }\h
\end{align*}
since $\psi ^{i}\in \Aut(\h^{i})$.
In either case, we obtain the desired conclusion, and so (2) is proved.

Finally, we prove (3).
By (1), $\varphi $ as in (3) is a mapping from $\Aut^{*}(\h)$ to $\mathcal{A}(\h)$.
Now we define a new mapping $\varphi'$ on $\mathcal{A}(\h)$ such that $\varphi'((\psi^i)_{i \in I}) = \psi $ where $\psi $ is defined from $(\psi^i)_{i \in I}$ by as in (\ref{K10}).
Then by (2), $\varphi'$ is a mapping from $\mathcal{A}(\h)$ to $\Aut^{*}(\h)$.
Furthermore, it follows from the definitions of $\varphi $ and $\varphi '$ that $\varphi'\circ \varphi=\id_{\Aut^{*}(\h)}$ and $\varphi \circ \varphi '=\id_{\mathcal{A}(\h)}$.
This implies that $\varphi $ is a bijection.
\end{proof}

\begin{proof}[Proof of Proposition~\ref{lem:Aut}]
Let $\psi \in \Aut^*(\h)$.
Note that $\psi=\id_{V(\h)}$ if and only if $\psi|_{\h^i} = \id_{V(\h^{i})}$ for all $i \in I$.
Since $(\psi|_{\h^i})_{i \in I} \in \mathcal{A}(\h)$ by Lemma~\ref{lem:decomp} (1), $\psi|_{V_{s}(\h^{i})}=\id_{V_{s}(\h^{i})}$ for all $i\in I$.
Thus it suffices to show that
\begin{align}
\psi|_{V_{f}(\h^{i})}=\id_{V_{f}(\h^{i})}.\label{eq-lem:Aut-1}
\end{align}
If $\h^i$ is not isomorphic to $\h_2$, then $\h^{i}$ has exactly one fat vertex, and so (\ref{eq-lem:Aut-1}) holds.
Thus we may assume that $\h^i\simeq \h_2$.
Write $V_s(\h^i) = \{x\}$ and $V_f(\h^i) = \{y_1, y_2\}$.
Since $\h$ is not isomorphic to $\h_2$, we may assume that $y_{1}\in V_{f}(\h^{j})$ for some $j\in I\setminus \{i\}$, i.e., $y_{1}\in V_{f}(\h^{i})\cap V_{f}(\h^{j})$.
Then by Lemma~\ref{lem:fixfat}, $\psi (y_{1})=y_{1}$, and so (\ref{eq-lem:Aut-1}) holds.
\end{proof}

\begin{lem}
\label{lem:uni}
Let $\fH$ be a family of Hoffman graphs, and let $G$ be a slim $\fH$-line graph having a strict $\fH$-cover $\n$.
Then $G$ has a unique strict $\fH$-cover up to equivalence if the following hold:
\begin{enumerate}
\item \label{u1}
$\Aut^*(\n) = \{ \id _{V(\n)}\}$;
\item \label{u2}
$|\Aut(\n)|=|\Aut(G)|$; and
\item \label{u3}
every strict $\fH$-covers of $G$ is isomorphic to $\n$.
\end{enumerate}	
\end{lem}
\begin{proof}
Let $\m$ be a strict $\fH$-cover of $G$.
We show that $\m$ and $\n$ are equivalent.
By (\ref{u3}), $\m$ is isomorphic to $\n$.
This together with (\ref{u1}) and (\ref{u2}) implies that
\begin{align}
\Aut^*(\m) = \{ \id_{V(\m)} \}\label{K11}
\end{align}
and
\begin{align}
|\Aut(\m)|=|\Aut(G)|.\label{K12}
\end{align}

We define the homomorphism $r : \Aut(\m) \to \Aut(G)$ by $r(\varphi) = \varphi|_G$ for $\varphi \in \Aut(\m)$.
We claim that
\begin{align}
r \mbox{ is injective}.\label{K13}
\end{align}
Let $\varphi \in \Ker r$.
Since $\varphi \in \Aut(\m)$ and $\varphi|_{V_{s}(\m)} = \varphi|_G = r(\varphi) = \id_{V(G)}$, we have $\varphi \in \Aut^*(\m)$.
This together with (\ref{K11}) leads to $\varphi = \id_{V(\m)}$.
Since $\varphi $ is arbitrary, we have $\Ker r = \{ \id_{V(\m)} \}$, and so (\ref{K13}) holds.
By (\ref{K12}) and (\ref{K13}), we see that $r$ is bijective.

Take an isomorphism $\psi$ from $\n$ to $\m$.
Since $(\psi |_G)^{-1} \in \Aut(G)$ and $r$ is bijective, there exists $\sigma \in \Aut(\m)$ such that $r(\sigma )=(\psi |_G)^{-1}$.
Then we can verify that $\sigma \circ \psi$ is an isomorphism from $\n$ to $\m$ and its restriction to $G$ is equal to $\id_{V(G)}~(=\id_{V_{s}(\n)})$.
Therefore $\m$ and $\n$ are equivalent.	
\end{proof}

\begin{proof}[Proof of Proposition~\ref{lem:pc}]
Let $G\in \cY$.
By (I) and the definition of $\Phi $, there exists a strict $\fH$-cover $\n \in \cX$ of $G$.
Now we show that $G$ and $\n$ satisfy the three conditions in Lemma~\ref{lem:uni}.
The conditions (\ref{u1}) and (\ref{u2}) in Lemma~\ref{lem:uni} follow from (II) and (IV), respectively.
By (I) and (III), $\Phi $ is bijective.
In particular, $\Phi ^{-1}(G)~(\in \cX)$ is a unique strict $\fH$-cover of $G$ up to isomorphism.
This implies that the condition (\ref{u3}) in Lemma~\ref{lem:uni} holds.
Hence it follows from Lemma~\ref{lem:uni} that $G$ has a unique strict $\fH$-cover up to equivalence, which proves the proposition.
\end{proof}

Now we explain our strategy:
Fix a family $\fH\subset \fO$ with $\h_2 \in \fH$ and a positive integer $N\geq 7$.
Since a graph is a slim $\fH$-line graph if and only if it is a slim $\bar{\fH}$-line graph, we may assume that $\fH=\bar{\fH}$.
Then we can construct the families $\cX$ and $\cY$ in Proposition~\ref{lem:pc} using computer programming.
Note that the conditions~(I) and (II) in Proposition~\ref{lem:pc} always hold by Lemma~\ref{lem:ExiCov} and Proposition~\ref{lem:Aut}, respectively.
Furthermore, we can judge whether the conditions~(III) and (IV) in Proposition~\ref{lem:pc} hold by computer search.
If these conditions are satisfied, then $N_{\fH} \le N$.
Note that Proposition~\ref{prop:lb} gives the smallest case for $N$.

Indeed, when $\h_{2}\in \fH \subset \fO$ and $N_{\fH}$ is small, some softwares such as MAGMA~\cite{MAGMA} can find an integer $N$ as in Theorem~\ref{thm:uni} along above strategy.
Recall that Taniguchi~\cite{T1} found $N$ in Theorem~\ref{thm:T1} as $N=8$ by computer search.
Similarly, we seek an integer $N$ as in Theorem~\ref{thm:uni} for $\fH$ other than $\{\h_{2},\h_{5}\}$ as follows:

\begin{ex}
\label{ex: 6.7}
Let $\h_5'$ be the Hoffman graph obtained from three independent (slim) vertices by joining a new fat vertex (see Figure~\ref{fig:h5'}).
\begin{figure}	[htbp]
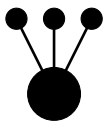
	\label{fig:h5'}
	\centering
	\Hoffee{20pt}
	\caption{The Hoffman graph $\h'_5$}
\end{figure}
Then Proposition~\ref{prop:lb} leads to $N_{\{\h_2,\h_5'\}} \geq 2 \cdot 3 + 2=8$, and by computer search, we obtain that $N_{\{\h_2,\h_5'\}}$ is actually equal to $8$. 
\end{ex}

\section*{Acknowledgements}
We would like to express our sincere gratitude to Professor Munemasa for his helpful comments.
	
\bibliography{main}

\end{document}

%% file: HoffmanGraphs.tex
\tikzset{
			e/.style={line width=1 pt,
				line join=round,
				line cap=round},
			edot/.style={dotted,
				line width=1 pt,
				line join=round,
				line cap=round},
			wslim/.style={
				circle,
				fill=white,
				draw=black,
				minimum size=11 pt,
				inner sep=1pt},
			dot/.style={dotted,
				line width=1 pt,
				line join=round,
				line cap=round},
			lodot/.style={loosely dotted,
				line width=2 pt	,
				line join=round,
				line cap=round}
}
\newcommand{\Hoffa}[1]{
	\begin{tikzpicture}[]
		\node at (0,0) [fat] (a) {};
		\node at (0,1) [slim] (b) {};
		\node at (0,-0.8) [] {$\h_1$};
		\draw (a.center)[e]--(b.center);
	\end{tikzpicture}
}
\newcommand{\Hoffb}[1]{
	\begin{tikzpicture}[]
		\node at (0,0) [fat] (a) {};
		\node at (0.5,1) [slim] (b) {};
		\node at (1,0) [fat] (c) {};
		\node at (0.5,-0.8) [] {$\h_2$};
		\draw [e](a.center)--(b.center)--(c.center);
	\end{tikzpicture}
}
\newcommand{\Hoffc}[1]{
	\begin{tikzpicture}[]
		\node at (0,1) [slim] (a) {};
		\node at (0.5,0) [fat] (b) {};
		\node at (1,1) [slim] (c) {};
		\node at (0.5,-0.8) [] {$\h_3$};
		\draw [e](a.center)--(b.center)--(c.center);
	\end{tikzpicture}
}
\newcommand{\Hoffd}[1]{
	\begin{tikzpicture}[]
		\node at (0,0) [fat] (a) {};
		\node at (0,1) [slim] (b) {};
		\node at (1,1) [slim] (d) {};
		\node at (1,0) [fat] (c) {};
		\node at (0.5,-0.8) [] {$\h_4$};
		\draw [e](a.center)--(b.center)--(d.center)--(c.center);
	\end{tikzpicture}
}
\newcommand{\Hoffe}[1]{
	\begin{tikzpicture}[]
		\node at (0,1) [slim] (a) {};
		\node at (0.5,0) [fat] (b) {};
		\node at (0.5,1) [slim] (c) {};
		\node at (1,1) [slim] (d) {};
		\node at (0.5,-0.8) [] {$\h_5$};
		\draw [e](a.center)--(b.center)--(c.center)--(d.center)--(b.center);
	\end{tikzpicture}
}

\newcommand{\Hoffee}[1]{
	\begin{tikzpicture}[]
		\node at (0,1) [slim] (a) {};
		\node at (0.5,0) [fat] (b) {};
		\node at (0.5,1) [slim] (c) {};
		\node at (1,1) [slim] (d) {};
		\draw [e](a.center)--(b.center)--(c.center) (d.center)--(b.center);
	\end{tikzpicture}
}

\newcommand{\Hofff}[1]{
	\begin{tikzpicture}[]
		\node at (0,0.5) [fat] (a) {};
		\node at (0.5,0) [slim] (b) {};
		\node at (0.5,1) [slim] (c) {};
		\node at (1,0.5) [slim] (d) {};
		\node at (1.5,0.5) [fat] (e) {};
		\node at (0.75,-0.8) [] {$\h_7$};
		\draw [e](a.center)--(b.center)--(c.center)--cycle;
		\draw [e](b.center)--(d.center)--(c.center) (d.center)--(e.center);
	\end{tikzpicture}
}
\newcommand{\Hoffh}[1]{
	\begin{tikzpicture}[]
		\node at (0,0.5) [fat] (a) {};
		\node at (0.5,0) [slim] (b) {};
		\node at (0.5,1) [slim] (c) {};
		\node at (1,0) [slim] (d) {};
		\node at (1,1) [slim] (e) {};
		\node at (1.5,0) [fat] (f) {};
		\node at (1.5,1) [fat] (g) {};
		\node at (0.75,-0.8) [] {$\h_9$};
		\draw [e](a.center)--(b.center)--(d.center)--(f.center);
		\draw [e](a.center)--(c.center)--(e.center)--(g.center);
		\draw [e](d.center)--(e.center);
	\end{tikzpicture}
}
\newcommand{\ExSum}{
	\begin{tikzpicture}[
		xscale=0.8,
		yscale=0.8,
		baseline=20pt]
		\node at (0,1) [slim_small] (a) {};
		\node at (1,1) [fat_small] (B) {};
		\node at (0.8,2) [slim_small] (d) {};
		\node at (1.2,2) [slim_small] (e) {};
		\node at (2,1) [slim_small] (f) {};
		\node at (3,1) [fat_small] (G) {};
		\node at (3.8,1.3) [slim_small] (h) {};
		\node at (4.6,1.6) [fat_small] (M){};
		\draw [e](a.center)--(B.center)--(f.center)--(G.center)--(h.center);
		\draw [e](d.center)--(B.center)--(e.center);
		\draw [e](h.center)--(M.center);
		\draw [dot](B.center)--(1,0);
		\draw [dot](2,2)--(B.center)--(0,2);
		\draw [dot] (G.center)--(3,0);
		\draw [dot](2.3,2)--(G.center);
	\end{tikzpicture}
}
\newcommand{\ExSumWithEdges}{
	\begin{tikzpicture}[
		xscale=0.8,
		yscale=0.8,
		baseline=20pt]
		\node at (0,1) [slim_small] (a) {};
		\node at (1,1) [fat_small] (B) {};
		\node at (0.8,2) [slim_small] (d) {};
		\node at (1.2,2) [slim_small] (e) {};
		\node at (2,1) [slim_small] (f) {};
		\node at (3,1) [fat_small] (G) {};
		\node at (3.8,1.3) [slim_small] (h) {};
		\node at (4.6,1.6) [fat_small] (M){};
		\draw [e](a.center)--(B.center)--(f.center)--(G.center)--(h.center);
		\draw [e](d.center)--(B.center)--(e.center);
		\draw [e](h.center)--(M.center);
		\draw [e] (a.center)--(d.center) (a.center)--(e.center);
		\draw [e] (f.center)--(d.center) (f.center)--(e.center);		
		\draw [e] (a) to [bend right=60](f);
		\draw [e] (f) to [bend left=60](h);
	\end{tikzpicture}
}
\newcommand{\ExSuma}{
	\begin{tikzpicture}[
		xscale=0.8,
		yscale=0.8,
		baseline=20pt]
		\node at (0,1) [slim_small] (a) {};
		\node at (1,1) [fat_small] (B) {};
		\draw [e](a.center)--(B.center);
	\end{tikzpicture}
}
\newcommand{\ExSumb}{
	\begin{tikzpicture}[
		xscale=0.8,
		yscale=0.8,
		baseline=20pt]
		\node at (1,1) [fat_small] (B) {};
		\node at (0.8,2) [slim_small] (d) {};
		\node at (1.2,2) [slim_small] (e) {};
		\draw [e](d.center)--(B.center)--(e.center);
	\end{tikzpicture}
}
\newcommand{\ExSumc}{
	\begin{tikzpicture}[
		xscale=0.8,
		yscale=0.8,
		baseline=20pt]
		\node at (1,1) [fat_small] (B) {};
		\node at (2,1) [slim_small] (f) {};
		\node at (3,1) [fat_small] (G) {};
		\draw [e](B.center)--(f.center)--(G.center);
	\end{tikzpicture}
}
\newcommand{\ExSumd}{
	\begin{tikzpicture}[
		xscale=0.8,
		yscale=0.8,
		baseline=20pt]
		\node at (3,1) [fat_small] (G) {};
		\node at (3.8,1.3) [slim_small] (h) {};
		\node at (4.6,1.6) [fat_small] (M){};
		\draw [e](G.center)--(h.center);
		\draw [e](h.center)--(M.center);
	\end{tikzpicture}
}
\newcommand{\ExEquiv}{
	\begin{tikzpicture}[
		xscale=1,
		yscale=1.3,
		baseline=20pt]
		\node at (1,-0.3) [] () {$G$};
		\node at (0.5,0.5) [slim,label=west:$x$] (x) {};
		\node at (1.5,0.5) [slim,label=east:$y$] (y) {};
		\draw [e] (x.center)--(y.center);
	\end{tikzpicture}
}
\newcommand{\ExEquiva}{
	\begin{tikzpicture}[
		xscale=1,
		yscale=1.3,
		baseline=20pt]
		\node at (1,-0.3) [] () {$\h$};
		\node at (0,0) [fat] (A) {};
		\node at (1,1) [fat] (B) {};
		\node at (0.5,0.5) [slim,label=west:$x$] (x) {};
		\node at (1.5,0.5) [slim,label=east:$y$] (y) {};
		\draw [e] (A.center)--(x.center)--(B.center)--(y.center);
		\draw [e] (x.center)--(y.center);
		\draw [dot] (1,0.2)--(B.center);
	\end{tikzpicture}
}
\newcommand{\ExEquivb}{
	\begin{tikzpicture}[
		xscale=1,
		yscale=1.3,
		baseline=20pt]
		\node at (1,-0.3) [] () {$\h'$};
		\node at (2,0) [fat] (C) {};
		\node at (1,1) [fat] (B) {};
		\node at (0.5,0.5) [slim,label=west:$x$] (x) {};
		\node at (1.5,0.5) [slim,label=east:$y$] (y) {};
		\draw [e] (x.center)--(B.center)--(y.center)--(C.center);
		\draw [e] (x.center)--(y.center);
		\draw [dot] (1,0.2)--(B.center);
	\end{tikzpicture}
}

\newcommand{\ExOfTilde}{
	\begin{tikzpicture}[
		xscale=0.8,
		yscale=0.8,
		baseline=17pt ]
		\node at (0,1) [slim] (a) {};
		\node at (1,1) [fat] (B) {};
		\node at (1,0) [slim] (c) {};
		\node at (0.8,2) [slim] (d) {};
		\node at (1.2,2) [slim] (e) {};
		\node at (2,1) [slim] (f) {};
		\node at (3,1) [fat] (G) {};
		\node at (3.8,1.3) [slim] (h) {};
		\node at (2.6,0) [slim] (i) {};
		\node at (3,0) [slim] (j) {};
		\node at (3.4,0) [slim] (k) {};
		\draw [e](a.center)--(B.center)--(f.center)--(G.center)--(h.center);
		\draw [e](d.center)--(B.center)--(e.center);
		\draw [e](B.center)--(c.center);
		\draw [e](i.center)--(G.center)--(j.center);
		\draw [e](G.center)--(k.center)--(j.center);
		\draw [dot](0,2)--(B.center)--(0,0);
		\draw [dot](2,2)--(B.center)--(2,0);
		\draw [dot](3,0)--(G.center)--(4,0.2);
		\draw [dot](2.5,2)--(G.center)--(2,0);
	\end{tikzpicture}
}
\newcommand{\ExOfTildee}{
	\begin{tikzpicture}[
		xscale=0.8,
		yscale=0.8,
		baseline=20pt]
		\node at (0,1) [slim] (a) {};
		\node at (1,1) [fat] (B) {};
		\node at (1,0) [slim] (c) {};
		\node at (0.8,2) [slim] (d) {};
		\node at (1.2,2) [slim] (e) {};
		\node at (2,1) [slim] (f) {};
		\node at (3,1) [fat] (G) {};
		\node at (3.8,1.3) [slim] (h) {};
		\node at (2.6,0) [slim] (i) {};
		\node at (3,0) [slim] (j) {};
		\node at (3.4,0) [slim] (k) {};
		\node at (-1,1) [fat] (L){};
		\node at (4.6,1.6) [fat] (M){};
		\node at (1,-1) [fat] (N){};
		\draw [e](L.center)--(a.center)--(B.center)--(f.center)--(G.center)--(h.center);
		\draw [e](d.center)--(B.center)--(e.center);
		\draw [e](B.center)--(c.center)--(N.center);
		\draw [e](i.center)--(G.center)--(j.center) (h.center)--(M.center);
		\draw [e](G.center)--(k.center)--(j.center);
		\draw [dot](0,2)--(B.center)--(0,0);
		\draw [dot](2,2)--(B.center)--(2,0);
		\draw [dot](3,0)--(G.center)--(4,0.2);
		\draw [dot](2.5,2)--(G.center)--(2,0);
	\end{tikzpicture}
}

